\documentclass[sn-mathphys-num]{sn-jnl}
\usepackage{amsmath,amsthm,amssymb,amsfonts}
\usepackage{mathtools} 
\usepackage{bm}
\usepackage{mathrsfs} 
\usepackage{changes} 
\usepackage{url} 
\usepackage{enumitem}
\usepackage[activate={true,nocompatibility},final,tracking=true,kerning=true,
spacing=true,factor=1100,stretch=10,shrink=10]{microtype} 
\usepackage[utf8]{inputenc}
\usepackage[english]{babel}

\usepackage{graphicx}
\usepackage{longtable} 
\usepackage{caption}
\usepackage{tikz}
\usepackage{subfig}
\usepackage{pgfplots}
\usepackage{color}
\usepackage{array}
\usepackage{booktabs}
\usepackage{hyperref}

\usepackage{amsmath}

\usepackage{algorithmic}
\usepackage[ruled]{algorithm2e}
\usepackage{setspace}
\SetAlCapSkip{0.5em}
\SetKwInOut{Input}{Input}
\SetKwInOut{Output}{Output\,}
\SetKwInOut{Data}{Data}
\SetKwProg{Tree}{Tree}{}{EndTree}

\newcommand{\N}{\mathbb{N}}

\newcommand{\R}{\mathbb{R}}

\newcommand{\cC}{\mathcal{C}}
\newcommand{\cA}{\mathcal{A}}
\newcommand{\cM}{\mathcal{M}}

\newcommand{\inner}[2]{\langle #1,#2\rangle} 

\newcommand{\norm}[1]{\|{#1}\|}

\newcommand{\abs}[1]{|#1|}

\newcommand{\tu}{{\tilde u}}
\newcommand{\DD}{\mathcal{D}}
\newcommand{\HH}{{\mathcal{H}}} 

\newcommand{\wto}{\rightharpoonup}

\newcommand{\parcomp}{\triangleright}
\newcommand{\mgap}{\vspace{.1in}}

\newcommand{\defeq}{\stackrel{\text{def}}{=}} 

\newcommand{\set}[1]{\{#1\}}

\newcommand{\seq}[1]{\left\{#1\right\}}

\newcommand{\lab}[1]{\label{#1}}

\DeclareMathOperator{\Dom}{Dom}

\DeclareMathOperator{\Gra}{Gra}
\DeclareMathOperator{\dom}{dom}

\DeclareMathOperator{\prox}{prox}

\DeclareMathOperator{\Sign}{Sign}
\DeclareMathOperator{\epi}{epi}

\newtheorem{theorem}{Theorem}[section]
\newtheorem{lemma}[theorem]{Lemma}

\newtheorem{proposition}[theorem]{Proposition}
\newtheorem{remark}[theorem]{Remark}
\newtheorem{definition}[theorem]{Definition}

\definechangesauthor[name={M. Marques Alves}, color=blue]{M}
\definechangesauthor[name={Dirk}, color=green!50!black]{D}

\title{
  A general framework for inexact splitting algorithms with relative errors and applications to Chambolle-Pock and Davis-Yin methods
}
\author[1]{\fnm{M.} \sur{Marques Alves}}\email{maicon.alves@ufsc.br}

\author[2]{\fnm{Dirk A.} \sur{Lorenz}}\email{d.lorenz@uni-bremen.de}

\author[3]{\fnm{Emanuele} \sur{Naldi$^*$}}\email{emanuele.naldi@edu.unige.it}

\affil[1]{Department of Mathematics,
Federal University of Santa Catarina,
88040-900 Florian\'opolis, Brazil}
\affil[2]{Faculty 3 - Mathematics and Computer Science, Center for Industrial Mathematics, University of Bremen, Postfach 330440, 28334 Bremen,
  Germany}
\affil[3]{MaLGa - DIMA, University of Genoa, 16146 Genoa, Italy}

\begin{document}

\abstract{%
In this work we apply the recently introduced framework of degenerate preconditioned proximal point algorithms to the hybrid proximal extragradient (HPE) method for maximal monotone inclusions. The latter is a method that allows inexact proximal (or resolvent) steps where the error is controlled by a relative-error criterion. Recently the HPE framework has been extended to the Douglas-Rachford method by Eckstein and Yao. In this paper we further extend the applicability of the HPE framework to splitting methods. To this end we use the framework of degenerate preconditioners that allows to write a large class of splitting methods as preconditioned proximal point algorithms. In this way, we modify many splitting methods such that one or more of the resolvents can be computed inexactly with an error that is controlled by an adaptive criterion. Further, we illustrate the algorithmic framework in the case of Chambolle-Pock's primal dual hybrid gradient method and the Davis-Yin's forward Douglas-Rachford method. In both cases, the inexact computation of the resolvent shows clear advantages in computing time and accuracy.
}



\maketitle

\section{Introduction}
\label{sec:intro}

In this paper we consider maximal monotone inclusions in a real Hilbert space $\HH$, i.e. inclusions of the type 
\begin{align}
\label{eq:prob}
0\in \cA u
\end{align}
where $\cA:\HH\to 2^{\HH}$ is a maximal monotone operator. 
Our goal is to find solutions to these inclusions, i.e. zeros of the operator $\cA$. These problems have a long history and many other problems are covered by this general formulation, e.g. convex optimization problems, convex-concave saddle-point problems 
and variational inequalities (see, e.g, \cite{BCombettes}), but also split inclusions of the 
type $0\in (A_{1} + \cdots + A_{N})x$, involving many maximal monotone operators, can be rewritten as an inclusion with a single operator $\cA$ in a larger space, see~\cite{BCLN22}.

A famous meta algorithm for monotone inclusions is the proximal point method~\cite{rockafellar1976} 
 that iterates the resolvent $J_{\cA} = (I+\cA)^{-1}$, i.e. it computes the sequence defined by $u^{k+1} = J_{\cA} u^{k}$, for an initialization $u^{0}$.
Resolvents are not always simple to compute, and if this is the case, the hybrid proximal extragradient (HPE) method~\cite{solodov1999} provides a flexible framework to compute inexact resolvents with a relative error-criterion. The goal of the paper is to extend the HPE framework to the so-called preconditioned proximal point method with degenerate preconditoner as it has been proposed in~\cite{BCLN22}. It has been shown in~\cite{BCLN22} that this framework allows a concise treatment of many existing splitting methods and we will extend the HPE framework to that case.

We will show that this allows us to develop splitting methods with inexact evaluation of resolvents where the errors are controlled by a relative-error criterion. Applying the method to the case of the Douglas-Rachford splitting operator we will obtain exactly the method recently proposed by Eckstein and Yao in~\cite{eckstein2018}. To show the flexibility of the method we also derive inexact versions of the Chambolle-Pock method~\cite{chambolle2011} and the Davis-Yin (forward Douglas-Rachford) method~\cite{davis2017}.

More precisely, as an application of our general degenerate preconditioned HPE framework for solving \eqref{eq:prob}, we consider the more structured inclusions
\begin{align} \lab{eq:prob02}
0\in A_1x + A_2x,
\end{align}
on a real Hilbert space $H$, a version of \eqref{eq:prob02} allowing compositions with linear operators
\begin{align} \lab{eq:prob03}
0\in A_1x + K^*A_2Kx,
\end{align}
as well as a three-operator case involving an additional cocoercive operator $B \colon H\to H$, namely
\begin{align} \lab{eq:prob04}
0\in A_1x +A_2x + Bx,
\end{align}
where in all instances $A_1$ and $A_2$ denote set-valued maximal monotone operators, $K$ is a bounded linear operator and
$B$ is a point-to-point cocoercive operator; more details below in Subsection \ref{sec:notation}. 
For instance, problem \eqref{eq:prob03} can be written in the format of \eqref{eq:prob} using $u = (x, y)\in H^2=\HH$
and $\cA:\HH\to 2^\HH$ defined by $\cA u = \left(A_1x + K^*y, -Kx + A_2^{-1}y\right)$ or, more shortly, 
\begin{align}\lab{def:T}
 \cA: =
 \begin{bmatrix}
 \phantom{-}A_1 & K^* \\
-K  & A_2^{-1}
\end{bmatrix}.
\end{align}
Using a similar reasoning, one can also recast \eqref{eq:prob04} into the general framework of monotone inclusions for a single maximal monotone operator (see \cite{BCLN22}).

More important than the fact that \eqref{eq:prob02} -- \eqref{eq:prob04} are special instances of \eqref{eq:prob} is the fact that we are able to explicitly define (degenerate) preconditioners allowing us to design and analyze practical (relative-error) inexact operator splitting algorithms for solving the above structured inclusions. As we mentioned before, we will propose and study the asymptotic convergence of a degenerate preconditioned HPE method, combining ideas from the HPE theory and the more recent degenerate preconditioned proximal point algorithm~\cite{BCLN22}.

\paragraph{The hybrid proximal extragradient (HPE) method.} 
In the seminal paper~\cite{rockafellar1976}, Rockafellar showed that
if at the current iterate $u^k$ the next one, namely $u^{k+1}$, is computed satisfying the \emph{summability} condition
\begin{align}\label{eq:inexact.prox}
 \sum_{k = 0}^\infty\,\left\|u^{k+1} - (I +\lambda_k \cA)^{-1}u^k\right\|< \infty
\end{align}
and $\seq{\lambda_k}$ is bounded away from zero, then the sequence of approximations $\seq{u^k}$ converges (weakly) 
to a solution of \eqref{eq:prob} (assuming that there exists at least one).
As an alternative to \eqref{eq:inexact.prox}, some modern inexact versions of the proximal point algorithm employ \emph{relative-error tolerances} for solving subproblems, namely to compute 
$u^{k+1}\approx (I +\lambda_k \cA)^{-1}u^k$.
The first methods of this type were proposed by Solodov and Svaiter in~\cite{solodov1999b, solodov2001, solodov1999} and subsequently studied, e.g., in~\cite{alves2016, monteiro2010, monteiro2012, monteiro2013}.
The main idea consists in decoupling the exact proximal point iteration 
$u^{k+1} = (I + \lambda_k \cA)^{-1}u^k$ as
\begin{align} \label{eq:dec.prox}
  v^{k+1} \in \cA u^{k+1},\quad  \lambda_k v^{k+1} + u^{k+1} - u^k = 0,
\end{align}
and then relaxing \eqref{eq:dec.prox}
within relative-error tolerance criteria.
The hybrid proximal extragradient (HPE) method \cite{solodov1999} has been shown to be
very effective as a framework for the design and analysis
of many concrete algorithms and can be described (in its simplest form, without $\varepsilon$-enlargements) as follows: 
for all $k\geq 0$,
\begin{align}
\left\{
\begin{array}{ll}
  \label{eq:v.hpe}
  v^{k+1}\in \cA \tu^{k+1},\quad  
 \norm{\lambda_k v^{k+1} + \tu^{k+1} - u^k} \leq \sigma \norm{\tu^{k+1} - u^k},\\[3mm]
  u^{k+1} = u^k - \lambda_k v^{k+1},\\
        \end{array}
        \right.
\end{align}
where $\sigma\in [0, 1)$.
If $\sigma=0$, then it is easy to see that \eqref{eq:v.hpe} reduces to the exact proximal point method
\eqref{eq:dec.prox}.
We also mention that the update rule for $u^{k+1}$ as in \eqref{eq:v.hpe}, 
namely, $u^{k+1} = u^k - \lambda_k v^{k+1}$, is exactly what we mean by an \emph{extragradient step} (this goes back to the seminal work of 
Korpolevich~\cite{korpelevich1976}, see also~\cite{nemirovski2005}).

As it was already mentioned, the main contribution of this paper is to propose and study a
degenerate preconditioned version of the HPE method \eqref{eq:v.hpe}.

\paragraph{The degenerate preconditioned proximal point method.} 
In~\cite{BCLN22}, the authors introduced a generalization of the proximal point method for solving \eqref{eq:prob} as follows:
for all $k\geq 0$,
\begin{align} \lab{eq:prepp}
u^{k+1} = (\cM+\cA)^{-1}\cM u^k,
\end{align}
where $\cM$ is a self-adjoint and positive \emph{semidefinite} bounded linear operator in $\HH$. By taking $\cM = I$ in \eqref{eq:prepp} we recover the proximal point iteration $u^{k+1} = (I + \cA)^{-1}u^k$
with $\lambda_k\equiv 1$.
In the case of positive definite $\cM$, the iteration~\eqref{eq:prepp} can be written as 
\begin{align*}
  u^{k+1} = (I + \cM^{-1}\cA)^{-1}u^{k}
\end{align*}
and hence, is exactly a proximal point iteration in the space where the inner product is changed to the $\cM$-inner product
\begin{align}
\label{eq:m-inner-product}
\inner{x}{y}_{\cM}\defeq \inner{x}{\cM y}.
\end{align}
And since $\cM^{-1}\cA$ is maximal monotone in the Hilbert space with the $\cM$-inner product, convergence of the method follows from standard results on the proximal point method.
The main point of the contribution~\cite{BCLN22} was to include in their analysis the case in which $\cM$ may have a nontrivial kernel, in this way opening the possibility of designing new operator-splitting methods. The asymptotic analysis of
\eqref{eq:prepp} was carried out by assuming that $(\cM+\cA)^{-1}$ is a (point-to-point) Lipschitz continuous operator in $\HH$. This assumption is satisfied by interesting examples of $\cM$ and $\cA$ in operator-splitting problems; for instance, for
$\cA$ as in \eqref{def:T}, considering
\begin{align}\lab{eq:def.m}
 \cM: =
 \begin{bmatrix}
 \frac{1}{\theta}I & -K^* \\
-K  & \frac{1}{\tau}I
\end{bmatrix}
,
\end{align}
(with $0<\theta\tau\leq \norm{K}^{-2}$), the generalized proximal point method \eqref{eq:prepp} leads to the primal-dual hybrid gradient method by Chambolle-Pock~\cite{chambolle2011} and the convergence follows (also in the edge case $\theta\tau=\norm{K}^{-2}$) by the theory developed in~\cite{BCLN22}.

\paragraph{The reduced method.}
The framework from~\cite{BCLN22} allows for degenerate preconditioners $\cM$, i.e. the operator $\cM$ is only positive semidefinite and can have a non-trivial kernel. This gives more flexibility but also  leads to so-called reduced methods which we illustrate here with an example.

In the case $K=I$ and $\tau=\theta=1$ in \eqref{def:T} and \eqref{eq:def.m} the involved operators become 
\begin{align*}
\cA =
  \begin{bmatrix}
    A_1 & I\\
    -I & A_2^{-1}
  \end{bmatrix}, \quad \cM =
  \begin{bmatrix}
    I & -I\\
    -I & I
  \end{bmatrix}.
\end{align*}
Here, the preconditioner $\cM$ has a large kernel, namely the span of the vectors of the form $(w,\ w)$. As has been shown in~\cite{BCLN22}, this can be used to derive a ``reduced'' algorithm: Instead of iterating $u^{k}\in \HH^{2}$ one can use the decomposition $\cM = \cC\cC^{*}$ with $\cC:\HH\to\HH^{2}$, $\cC w = (w,\ -w)$, introduce $w^k = \cC^* u^k$ for every $k$, and rewrite the iteration as 
\begin{align}\label{eq:reduced-PPP}
w^{k+1} = \cC^{*}(\cM+\cA)^{-1}\cC w^{k}.
\end{align}
Moreover, from the weak limit $w^{*}$ of this sequence one gets via $(\cM+\cA)^{-1}\cC w^{*}$ a fixed point of $\cA$ (cf.~\cite[Theorem 2.14]{BCLN22}). In the above case this reduced algorithm is exactly the famous Douglas-Rachford iteration for the split inclusion $0\in A_1 x + A_2 x$ (see~\cite[Section 3]{BCLN22}).\bigskip

The main goal of the present work is to generalize \eqref{eq:prepp} allowing errors in the computation of $(\cM+\cA)^{-1}$ within relative-error tolerances in the spirit of \eqref{eq:v.hpe}. 
Similarly to \eqref{eq:dec.prox}, the iteration \eqref{eq:prepp} can be decoupled as a inclusion/equation system leading to 
the new error criterion
\[
\cM v^{k+1}\in \cA \tu^{k+1},\quad  
 \norm{\lambda_k v^{k+1} + \tu^{k+1} - u^k}_{\cM} \leq \sigma \norm{\tu^{k+1} - u^k}_{\cM},
\]
where $\norm{\cdot}_{\cM}$ denotes the seminorm induced by $\cM$.

We apply our (relative-error) inexact degenerate preconditioned HPE framework for different instances of $\cA$ and $\cM$, obtaining in this way new flexible and efficient operator-splitting methods for \eqref{eq:prob02}, \eqref{eq:prob03} and
\eqref{eq:prob04} (more details below in 
Subsection \ref{sec:contribution}).

\subsection{Related work}
\label{sec:related-work}

Inexact steps for methods that solve monotone inclusions and which are based on resolvents have been investigated from different angles. 
One line of work models additive errors in each step and derive conditions under which one still obtains convergence. 
In \cite{combettes2004solving}, for example, it is shown that the proximal gradient method converges if one still allows for additive errors that are summable. A similar result for the Douglas-Rachford method can be found in~\cite{svaiter2011weak}. In both cases, one has to drive the error down to zero in a pre-designed manner. Moreover, this condition is hard to check in practice (one would need a way to calculate or estimate the error to the exact step). Another line of work uses relative error conditions and this approach dates back to~\cite{solodov1999} for the proximal point method. In these approaches the error condition only need quantities that can be computed from the iterates. Only recently, this approach has been extended to splitting methods, more specifically to the Douglas-Rachford method (and hence also to the alternating directions method of multipliers (ADMM)) in~\cite{alves2020, eckstein2018}.

Many splitting methods can be written in the framework of the proximal point method. For the Douglas-Rachford method this has been observed in~\cite{eckstein1992}. In~\cite{bredies2015preconditioned} it has been observed that a specific reformulation of the split inclusion $0\in (A_{1}+A_{2})x$ in a product space and the use of a \emph{degenerate} preconditioner for the proximal point method also leads to the Douglas-Rachford method (see also~\cite{BCLN22}). The paper~\cite{BCLN22} also showed that other splitting methods, involving any number of monotone operators, can be derived in this framework. In~\cite{bredies2022graph} it has then been shown that one can design splitting methods for any ``communication structure'' between the operators, i.e. one can prescribe in which order the resolvents of the operators are evaluated and how the results are passed among the operators. The communication structure can be encoded by a communication graph which is then used to design the respective algorithm. The methods introduced in~\cite{bredies2022graph} extended existing splitting algorithms such as the one introduced by Ryu in \cite{Ryu}, the parallel Douglas-Rachford introduced and analyzed in \cite{RyuYin_parallel, condat, Campoy2022} and the method introduced by Malitsky and Tam in \cite{MalitskyTam2023resolvent}. In~\cite{bredies2022graph} the authors also showed that all these algorithms (and more in general all the graph-based ones) are a particular case of degenerate preconditioned proximal point. In the present work, we introduce relative errors criteria for the degenerate preconditioned proximal point algorithm and thus, all the methods mentioned above can benefit from our analysis. The connection established in this work between the HPE framework and the degenerate preconditioning setting lays also the basis for future exploration of (degenerate) variable metric analysis \cite{lorenz2024}, potentially enabling the development of more adaptive and efficient splitting methods.

\subsection{Contribution}
\label{sec:contribution}

In this paper we make the following contributions:
\begin{itemize}
\item We extend the hybrid proximal extragradient method to proximal point methods with degenerate preconditioners and show weak convergence of the iterates.
\item We show how this approach can be applied to splitting methods, namely the Douglas-Rachford method (recovering the results by Eckstein and Yao from~\cite{eckstein2018}), the Chambolle-Pock method~\cite{chambolle2011} and the Davis-Yin method~\cite{davis2017}, thereby developing inexact versions of the algorithms with relative error conditions.
\item We illustrate in numerical examples that the proposed methods are effective and efficient when applied to practical problems.
\end{itemize}

\subsection{Notation}
\label{sec:notation}

Let $\HH$ be a real Hilbert space.
Let $T \colon \HH \to 2^\HH$ be a multivalued map. 
The \emph{effective domain} and \emph{graph} of $T$ are $\Dom T = \set{x  \mid Tx \neq \emptyset}$ and $\Gra T =\{(x, v) \mid v \in Tx\}$, respectively. 
The \emph{inverse} of $T \colon \HH \to 2^\HH$ is $T^{-1} \colon \HH \to 2^\HH$ defined at
any $x\in \HH$  by $v\in T^{-1}x$ if and only if $x\in Tv$. 
The \emph{sum} of two multivalued maps 
$T, S\colon \HH \to 2^\HH$  is 
$T + S \colon \HH \to 2^\HH$, defined by 
the usual Minkowski sum $(T + S)(x) = \set{u + v \mid u\in Tx, v\in Sx}$. 
For $\lambda > 0$, we also define 
$\lambda T \colon \HH \to 2^\HH$ by 
$(\lambda T)x = \lambda Tx = \set{ \lambda v \mid v \in Tx}$.
Whenever necessary, we will also identify single-valued maps $B \colon \Dom B\subset \HH \to \HH$ with its multivalued representation 
$B \colon \HH \to 2^\HH$ by $Bx = \set{Bx}$. 
A multivalued map $T \colon \HH \to 2^\HH$ is said to be a \emph{monotone operator} 
if $\inner{x - y}{u - v} \geq 0$ for all $(x, u)$, $(y, v) \in \Gra T$, and \emph{maximal monotone} if it is monotone and its graph $\Gra T$ is not properly contained in the graph of any other monotone operator on $\HH$.
A single-valued map 
$B \colon \Dom B\subset \HH \to \HH$ is monotone
if $\inner{x - y}{Bx - By}\geq 0$ for all $x, y \in \Dom B$. The \emph{resolvent} of a maximal monotone operator 
$T \colon \HH \to 2^\HH$ is $J_{T} = (I + T)^{-1}$, where $I$ denotes the identity operator in $\HH$. 
Let $f \colon \HH \to (-\infty, +\infty]$ be an extended real-valued function. The \emph{domain} and \emph{epigraph} of $f$
are $\dom f = \set{x \mid f(x) < +\infty}$ and $\epi f = \set{(x, \mu) \in \HH\times \R \mid \mu \geq f(x)}$, respectively.
Recall that $f$ is \emph{proper} if $\dom f\neq \emptyset$ and \emph{convex} (resp. \emph{lower semicontinuous}) if $\epi f$ is a convex (resp. closed) subset of $\HH\times \R$. 
The \emph{subdifferential} of $f$ is 
$\partial f \colon \HH \to 2^\HH$ defined by 
$\partial f(x) = \set{ v \mid f(y) \geq f(x) + \inner{y-x}{v}\;\; \mbox{for all}\; y \in \HH}$.
We also denote by $\Gamma_0(\HH)$ the set of all proper, convex and lower semicontinuous functions on $\HH$. We say that a map $f:\HH\to \R$ is $L$-\emph{smooth} if its gradient $\nabla f$ is Lipschitz continuous with constant $L$.

An operator $B:\HH\to\HH$ is \emph{$\beta$-cocoercive} if for all $x,y\in \HH$ it holds that 
\begin{align*}
\inner{Bx-By}{x-y} \geq \beta\norm{Bx-By}^{2}.
\end{align*}

For a linear, bounded, self-adjoint and positive semidefinite map $\cM\colon\HH\to\HH$ we call a decomposition $\cM = \cC\cC^{*}$ with $\cC$ linear, bounded and injective from a real Hilbert space $\DD$ to $\HH$ and \emph{onto decomposition} if $\cC^{*}$ is onto. Such a decomposition exists if $\cM$ has closed range (cf.~\cite[Proposition 2.3]{BCLN22}). With such an onto decomposition we have the identity 
\begin{align}
\label{eq:M-norm-C}
\norm{x}_{\cM}^{2} = \inner{x}{\cM x} = \inner{x}{\cC\cC^{*}x} = \norm{\cC^{*}x}^{2}.
\end{align}

For a multivalued map $A:\HH_{1}\to 2^{\HH_{1}}$ on a real Hilbert space $\HH_{1}$ and a linear and bounded operator $K:\HH_{1}\to\HH_{2}$ into another real Hilbert space the \emph{parallel composition} is $K\parcomp A = (KA^{-1}K^{*})^{-1}$. It holds that $K\parcomp A$ is a multivalued map on $\HH_{2}$ and it is monotone if $A$ is monotone. Maximal monotonicity of $K\parcomp A$ can be guaranteed under additional assumptions, cf.~\cite[Proposition 25.41]{BCombettes}.

\section{The hybrid proximal extragradient method with degenerate preconditioning}
\label{sec:hpe}

In this section we introduce and study the hybrid proximal extragradient method with a degenerate preconditioner. We consider a maximal monotone operator $\cA$ defined on a real Hilbert space $\HH$ and a linear operator $\cM$ which is \emph{admissible} in the sense of the following definition.
\mgap
\begin{definition}[\cite{BCLN22}]
An admissible preconditioner for the multivalued operator $\cA:\HH \to 2^{\HH}$ is a linear, bounded, self-adjoint and positive semidefinite operator
$\cM: \HH \to \HH$ such that
\[
 (\cM + \cA)^{-1}\cM\quad \mbox{is single-valued and has full domain}.
\]
\end{definition}

As noted in \cite{BCLN22}, this condition on the preconditioner $\cM$ is quite mild, in particular for applications to splitting methods where it is always satisfied (see also \cite{naldi2024thesis} for further details).\\
For a given initial value $u^{0}$, a constant $\sigma\in[0,1)$ and stepsizes $\lambda_{k}>0$ we consider sequences $\seq{u^{k}},\seq{\tu^{k}}, \seq{v^{k}}$ that fulfill 
\begin{align}
  \cM v^{k+1} & \in \cA \tu^{k+1},\label{eq:HPE-inclusion}\\
  \norm{\lambda_{k+1}v^{k+1} + \tu^{k+1} - u^{k}}_{\cM} & \leq \sigma\norm{\tu^{k+1}-u^{k}}_{\cM}, \label{eq:HPE-error-condition}\\
  u^{k+1} & = u^k - \lambda_{k+1}v^{k+1}. \label{eq:HPE-update} 
\end{align}

Before we analyze these sequences, let us provide some intuition about the iteration. 

\mgap

\begin{remark}
    In the case $\sigma = 0$, the inequality~\eqref{eq:HPE-error-condition} says that
\begin{align*}
  \cM (\lambda_{k+1}v^{k+1} + \tu^{k+1} - u^{k})=0
\end{align*}
which, together with the inclusion~\eqref{eq:HPE-inclusion}, shows that 
\begin{align*}
\tu^{k+1} \in (\cM+\lambda_{k+1}\cA)^{-1} \cM u^{k}.
\end{align*}
Moreover, we get that the update~\eqref{eq:HPE-update} is (after application of $\cM$)
\begin{align*}
  \cM u^{k+1} & = \cM u^{k} - \lambda_{k+1}\cM v^{k+1} = \cM \tu^{k+1}.
\end{align*}
This shows that the sequence 
$\seq{u^{k}}$ and $\seq{\tu^{k}}$ only differ in the kernel of $\cM$ and convergence of $\seq{u^{k}}$ or $\seq{\tu^{k}}$ may not be true in the general case. However, we will show below that the sequence $\seq{\tu^{k}}$ is indeed bounded and weakly converges to a zero of $\cA$.
\end{remark}

\mgap

The next two proposition provide some foundamental properties of the sequences generated by \eqref{eq:HPE-inclusion}-~\eqref{eq:HPE-update}.

\mgap

\begin{proposition}[Fundamental estimates]
  \label{prop:fundamental-estimates-HPE}
  Let $\cA$ be maximal monotone on $\HH$, $u^{*}$ a zero of $\cA$, $\cM$ linear, bounded and positive semidefinite on $\HH$ and $u^0\in \HH$.
  Consider sequences $\seq{u^{k}},\seq{\tu^{k}},\seq{v^{k}}$ fulfilling \eqref{eq:HPE-inclusion}-~\eqref{eq:HPE-update} for $\lambda_{k}>0$ and $\sigma\in[0,1)$. Then, for all $k \geq 0$,
  \begin{enumerate}[label=\roman*)]
  \item $(1-\sigma)\norm{\tu^{k+1} - u^{k}}_{\cM}\leq \norm{\lambda_{k+1}v^{k+1}}_{\cM} \leq (1+\sigma)\norm{\tu^{k+1}-u^{k}}_{\cM}$,
  \item $\norm{u^{k+1}-u^{*}}_{\cM}^{2} + (1-\sigma^{2})\norm{u^{k}-\tu^{k+1}}_{\cM}^{2} \leq \norm{u^{k}-u^{*}}_{\cM}^{2}$,
  \item $\norm{u^{k+1} - u^{*}}_{\cM}^{2} + (1-\sigma^{2}) \sum\limits_{i=0}^k\norm{u^{i}-\tu^{i+1}}_{\cM}^{2}\leq \norm{u^{0}-u^{*}}_{\cM}^{2}$.
  \end{enumerate}
  Consequently we also have for $k\to\infty$
  \begin{enumerate}[label=\roman*)]\setcounter{enumi}{3}
  \item $\norm{u^{k}-\tu^{k+1}}_{\cM} \to 0,\quad \norm{\lambda_{k}v^{k}}_{\cM}\to 0,\quad \norm{u^{k}-\tu^{k}}_{\cM}\to 0$.
  \end{enumerate}
\end{proposition}
\begin{proof}
  The error condition~\eqref{eq:HPE-error-condition} and the triangle inequality for the $\cM$-seminorm give
\begin{align*}
  \norm{\lambda_{k+1}v^{k+1}}_{\cM} & \leq \norm{\lambda_{k+1}v^{k+1} + \tu^{k+1} - u^{k}}_{\cM} + \norm{\tu^{k+1}- u^{k}}_{\cM} \\ & \leq (1+\sigma)\norm{\tu^{k+1}-u^{k}}_{\cM}
\end{align*}
which shows the right inequality in the first claim. For the left inequality we consider
\begin{align*}
  \norm{\tu^{k+1}-u^{k}}_{\cM} & \leq \norm{\lambda_{k+1}v^{k+1} + \tu^{k+1} - u^{k}}_{\cM} + \norm{\lambda_{k+1}v^{k+1}}_{\cM} \\
  & \leq  \sigma \norm{\tu^{k+1}-u^{k}}_{\cM} + \norm{\lambda_{k+1}v^{k+1}}_{\cM}
\end{align*}
which can be rearranged to the desired inequality.

  For the second claim we use the identity $\norm{a}_{\cM}^{2} - \norm{b}_{\cM}^{2} = \norm{a-b}_{\cM}^{2} +2\inner{a-b}{b}_{\cM}$ with 
$a=u^{k} - u^*$ and $b=u^{k+1} - u^*$ and get
\begin{align*}
\norm{u^{k} - u^*}_{\cM}^{2} - \norm{u^{k+1} - u^*}_{\cM}^{2} = \norm{u^{k}-u^{k+1}}_{\cM}^{2} + 2\inner{u^{k} - u^{k+1}}{u^{k+1} - u^*}_{\cM}.
\end{align*}
By repeating the same argument with $a=u^{k} - \tu^{k+1}$ and $b=u^{k+1} - \tu^{k+1}$ we also find
\begin{align*}
\norm{u^{k} - \tu^{k+1}}_{\cM}^{2} - \norm{u^{k+1} - \tu^{k+1}}_{\cM}^{2} & = \norm{u^{k}-u^{k+1}}_{\cM}^{2} \\ & \hspace{1cm} + 2\inner{u^{k} - u^{k+1}}{u^{k+1} - \tu^{k+1}}_{\cM}.
\end{align*}
Now by subtracting the latter from the former equality and using the update~\eqref{eq:HPE-update} we obtain
\begin{align*}
\norm{u^{k} - u^*}_{\cM}^{2} - \norm{u^{k+1} - u^*}_{\cM}^{2} &= \norm{u^{k} - \tu^{k+1}}_{\cM}^{2} - \norm{u^{k+1} - \tu^{k+1}}_{\cM}^{2}
\\ & \hspace{3cm} +2\inner{u^{k} - u^{k+1}}{\tu^{k+1} - u^*}_{\cM}\\
 & = \norm{u^{k} - \tu^{k+1}}_{\cM}^{2} - \norm{u^{k+1} - \tu^{k+1}}_{\cM}^{2}
\\ & \hspace{3cm} 
 +2\lambda_{k+1}\inner{v^{k+1}}{\tu^{k+1} - u^*}_{\cM}.
\end{align*}
Since $\cM v^{k+1}\in \cA \tu^{k+1}$ and 
$0\in \cA u^*$ one can use the monotonicity of $\cA$ and the definition of
$\inner{\cdot}{\cdot}_{\cM}$ to conclude that
\begin{align*}
\inner{v^{k+1}}{\tu^{k+1} - u^*}_{\cM} = \inner{\cM v^{k+1}}{\tu^{k+1} - u^*}\geq 0.
\end{align*}
Thus, we get from the update step~\eqref{eq:HPE-update} and the error condition~\eqref{eq:HPE-error-condition}
\begin{align*}
\norm{u^{k} - u^*}_{\cM}^{2} - \norm{u^{k+1} - u^*}_{\cM}^{2} &\geq \norm{u^{k} - \tu^{k+1}}_{\cM}^{2} - \norm{u^{k+1} - \tu^{k+1}}_{\cM}^{2}\\
 & = \norm{u^{k} - \tu^{k+1}}_{\cM}^{2} - \norm{\lambda_{k+1} v^{k+1} + \tu^{k+1} - u^{k}}_{\cM}^{2}\\
 & \geq (1-\sigma^2)\norm{u^{k}-\tu^{k+1}}_{\cM}^{2},
\end{align*}
which proves the desired inequality.

The third claim is a direct consequence of second by summing up and using a telescope sum.

The first two final convergence statements follow directly and for the last statement observing that by~\eqref{eq:HPE-error-condition} and~\eqref{eq:HPE-update} we have 
\begin{align*}
\norm{\tu^{k+1}-u^{k+1}}_{\cM}\leq 
\sigma\norm{\tu^{k+1} - u^{k}}_{\cM}
\end{align*}
which gives the desired convergence.
\end{proof}

\begin{proposition}[Boundedness and weak subsequential convergence]
  \label{prop:HPE-bounded-x-tilde}
  Let $\cA$ be maximal monotone on $\HH$, $u^{*}$ be a zero of $\cA$, $\cM$ an admissible preconditioner of $\cA$ with the onto decomposition $\cM=\cC\cC^{*}$. If $(\cM+\cA)^{-1}$ is Lipschitz continuous, $\seq{\tu^{k}}$ is generated by the iteration~\eqref{eq:HPE-inclusion}-~\eqref{eq:HPE-update} and $\inf_{k}\lambda_{k}>0$, then the sequence $\seq{\tu^{k}}$ is bounded and all weak cluster points are zeros of $\cA$.
\end{proposition}
\begin{proof}
  From~\eqref{eq:HPE-inclusion} we get that $\cM(v^{k+1} +\tu^{k+1}) \in (\cM+\cA)\tu^{k+1}$ which is equivalent to $\tu^{k+1} = (\cM+\cA)^{-1}\cM(v^{k+1} + \tu^{k+1})$. Since $u^{*}$ is a zero of $\cA$, it is also a fixed point of $(\cM+\cA)^{-1}\cM$. With this and the Lipschitz continuity of $(\cM+\cA)^{-1}$ (with constant $L$, say) we get (using~\eqref{eq:M-norm-C})
\begin{align*}
  \norm{\tu^{k+1}-u^{*}} & = \norm{(\cM+\cA)^{-1}\cM(v^{k+1} + \tu^{k+1}) - (\cM+\cA)^{-1}\cM u^{*}}\\
                               & \leq L \norm{\cC\cC^{*}(v^{k+1}+\tu^{k+1}-u^{*}}\\
                               & \leq L\norm{\cC}\norm{v^{k+1}+\tu^{k+1}-u^{*}}_{\cM}\\
  & \leq L\norm{\cC}\left(\norm{v^{k+1}}_{\cM} + \norm{\tu^{k+1}-u^{*}}_{\cM}\right).
\end{align*}
From Proposition~\ref{prop:fundamental-estimates-HPE} iii) we get that $\norm{\tu^{k+1}-u^{*}}_{\cM}$ is bounded and from iv) and $\inf_{k}\lambda_{k}>0$ we get $\norm{v^{k}}_{\cM}\to 0$, so the right hand side above is bounded.

Now let $\bar{u}$ be a weak cluster point of the sequence $\seq{\tu^{k}}$ and 
$\seq{\tu^{k_{j}}}$ be a subsequence which convergence weakly to $\bar{u}$. We have $\cM v^{k_{j}}\in \cA \tu^{k_{j}}$ and since $\norm{v^{k}}_{\cM}\to 0$ we also have $\norm{\cM v^{k}}\leq \norm{\cC}\norm{v^{k}}_{\cM}\to 0$. Since $\cA$ is maximal monotone, its graph is weakly-strongly closed in $\HH\times\HH$ and this proves that $0\in \cA \bar{u}$ as desired.
\end{proof}

We remark that we do not claim that the 
sequences $\seq{v^{k}}$ and $\seq{u^{k}}$ are bounded. In fact, they are not in general: In each iteration we can modify $v^{k}$ to $v^{k} + \tau v$ with some $v\neq 0$ with $\cC^{*}v = 0$ and fixed $\tau\neq 0$. This will not change the sequence $\seq{\tu^{k}}$, but clearly renders $\seq{v^{k}}$ unbounded. As $u^{k+1} = u^{0} - \sum\limits_{i=1}^{k+1}\lambda_iv^{i}$, the sequence $\seq{u^{k+1}}$ will be unbounded as well.

\mgap

\begin{lemma}[Opial property for seminorms]\label{Opial}
Let $\cM:\HH\to \HH$ be a linear bounded self-adjoint and positive semidefinite operator. Let $\seq{v^k}$ be a sequence weakly convergent to a point $v^*$, then 
	\[\liminf_{k\to \infty} \|v^k-v^*\|_{\cM}^2 = \liminf_{k\to \infty} \|v^k-v\|_{\cM}^2-\|v^*-v\|_{\cM}^2 \quad \text{ for all } v\in \HH.\]
	In particular, for every $v\in \HH \text{ with } \|v^*-v\|_{\cM}>0$, it holds 
	\[\liminf_{k\to \infty} \|v^k-v^*\|_{\cM} < \liminf_{k\to \infty} \|v^k-v\|_{\cM}.\]
\end{lemma}
\begin{proof}
	We have \[\|v^k-v\|_{\cM}^2 = \|v^k-v^*+v^*-v\|_{\cM}^2 = \|v^k-v^*\|_{\cM}^2+\|v^*-v\|_{\cM}^2+2\langle \cM(v^*-v) , v^k-v^* \rangle. \]
	Since $v^k \rightharpoonup v^*$, we have and $\langle \cM(v^*-v) , v^k-v^* \rangle\to 0$ for $k\to \infty$, and the claim follows.
\end{proof}

\mgap

Making use of Proposition~\ref{prop:fundamental-estimates-HPE}, Proposition~\ref{prop:HPE-bounded-x-tilde} and the previous Lemma, we are now ready to provide the first main result of the paper.

\mgap

\begin{theorem}[Weak convergence of the iterates]\label{thm:weak-conv-HPE}
  Let $\cA$ be maximal monotone on $\HH$ which has at least one zero, $\cM$ an admissible preconditioner of $\cA$ with the onto decomposition $\cM=\cC\cC^{*}$. If $(\cM+\cA)^{-1}$ is Lipschitz continuous, 
  $\seq{\tu^{k}}$ and $\seq{u^{k}}$ are generated by the iteration~\eqref{eq:HPE-inclusion}-\eqref{eq:HPE-update} and $\inf_{k}\lambda_{k}>0$, then the sequence $\seq{\tu^{k}}$ converges weakly to some $u^{*}$ with $0\in \cA u^{*}$.
\end{theorem}
\begin{proof}
  For any $x$ we have 
\begin{align*}
\norm{\tu^{k}-x}_{\cM}^{2} = \norm{\tu^{k}-u^{k}}_{\cM}^2 + \norm{u^{k}-x}_{\cM}^2 + 2\inner{\tu^{k}-u^{k}}{u^{k}-x}_{\cM}.
\end{align*}
If $x$ is a fixed point of $(\cM+\cA)^{-1}\cM$ we know from Proposition~\ref{prop:fundamental-estimates-HPE} ii) that $\norm{u^{k}-x}_{\cM}$ is decreasing and thus, has a limit, say $\ell(x)$.
The Cauchy-Schwarz inequality for the semi-inner product gives 
\begin{align*}
\inner{\tu^{k}-u^k}{u^k-x}_{\cM}\leq\norm{\tu^{k}-u^{k}}_{\cM}\norm{u^{k}-x}_{\cM}.
\end{align*}
Since $\norm{u^{k}-x}_{\cM}$ is bounded and by Proposition~\ref{prop:fundamental-estimates-HPE} iv) $\norm{\tu^{k}-u^{k}}_{\cM}\to 0$ we get that 
\begin{align*}
\inner{\tu^{k}-u^k}{u^k-x}_{\cM}\to 0.
\end{align*}
Since we already know that $\norm{\tu^{k}-u^{k}}_{\cM}^2\to 0$ we get that 
\begin{align*}
\norm{\tu^{k}-x}_{\cM}\to \ell(x).
\end{align*}
To conclude the proof we use the Opial property with respect to the seminorm $\norm{\cdot}_{\cM}$. With this we argue as follows: 
Since, by Proposition \ref{prop:HPE-bounded-x-tilde}, $\seq{\tu^{k}}$ is bounded, it has a weak cluster point $u^{*}$ and assume that $\tu^{k_{j}}\wto u^{*}$. By Proposition~\ref{prop:HPE-bounded-x-tilde} we know that $u^{*}$ is a fixed point of $(\cM+\cA)^{-1}\cM$. Consider a different weak cluster point $u^{**}$ which is a limit of another subsequence $\seq{\tu^{l_{j}}}$. If we assume that $\norm{u^{*}-u^{**}}_{\cM}>0$, then we get from Opial's Lemma \ref{Opial} both 
\begin{align*}
  \liminf_{j\to\infty}\norm{\tu^{k_{j}}-u^{*}}_{\cM} & < \liminf_{j\to\infty}\norm{\tu^{l_{j}}-u^{*}}_{\cM}\\
  \text{and}\ \liminf_{j\to\infty}\norm{\tu^{l_{j}}-u^{*}}_{\cM} & < \liminf_{j\to\infty}\norm{\tu^{k_{j}}-u^{*}}_{\cM}
\end{align*}
which is a contradiction. So we have $\norm{u^{*}-u^{**}}_{\cM}=0$ which means that $\cC^{*}u^{*}=\cC^{*}u^{**}$ and also $\cM u^{*}=\cM u^{**}$. But since $u^{*}$ and $u^{**}$ are both fixed points of $(\cM+\cA)^{-1}\cM$ we get
\begin{align*}
u^{*}  = (\cM+\cA)^{-1}\cM u^{*} = (\cM+\cA)^{-1}\cM u^{**} = u^{**}.
\end{align*}
In conclusion, the full sequence $\seq{\tu^{k}}$ converges weakly to a fixed point, as desired.
\end{proof}

Using the onto decomposition $\cM = \cC\cC^{*}$ we can obtain a reduced version of the preconditioned HPE method in the spirit of~\eqref{eq:reduced-PPP}. This iteration works as follows: Choose $w_{0}\in \DD$, $\sigma$ with $0\leq\sigma<1$ and stepsizes $\lambda_k>0$. Iteratively construct 
sequences $\seq{w^{k}}, \seq{z^{k}},\seq{\tu^{k}}$ which fulfill 
\begin{align}
  \cC z^{k+1} & \in \cA \tu^{k+1}\label{eq:reducedHPE-inclusion}\\
  \norm{\lambda_{k+1}z^{k+1}+\cC^{*}\tu^{k+1}-w^{k}} & \leq \sigma\norm{\cC^{*}\tu^{k+1}-w^{k}}\label{eq:reducedHPE-error-condition}\\
  w^{k+1} = w^{k}-\lambda_{k+1}z^{k+1}.\label{eq:reducedHPE-update}
\end{align}

\begin{algorithm}[t]
\setstretch{1.1}
	\begin{algorithmic}[1]
		\STATE \textbf{Initialize:} $w^0\in \DD$ and $\sigma \in [0,1)$ and stepsize sequence $\lambda_{k}>0$
		\FOR{$k=0,1, 2 \dots$}
                \STATE Produce a pair $(\tu^{k+1},z^{k+1})\in\HH\times\DD$ with
                \begin{align*}
                  \cC z^{k+1}\in \cA \tu^{k+1},\qquad \lambda_{k+1}z^{k+1}+\cC^*\tu^{k+1}\approx w^{k}.
                \end{align*}\vspace*{-1\baselineskip}
                \WHILE{$\|\lambda_{k+1}z^{k+1}+\cC^*\tu^{k+1}-w^{k}\|>\sigma\|\cC^*\tu^{k+1}-w^{k}\|$}
                \STATE Improve the pair $(\tu^{k+1},z^{k+1})\in\HH\times\DD$ with $\cC z^{k+1}\in \cA \tu^{k+1}$\\ to reduce $\norm{\lambda_{k+1}z^{k+1}+\cC^*\tu^{k+1} - w^{k}}$.
                \ENDWHILE
		\STATE Set $w^{k+1}=w^{k}-\lambda_{k+1}z^{k+1}$.
		\ENDFOR
	\end{algorithmic}
	\caption{Reduced version of the preconditioned HPE algorithm.}\label{alg:HPE_r}
\end{algorithm}

We have the following result.

\mgap

\begin{theorem}\label{thm:convergence_reduced}
  Let $\cA$ be maximal monotone on $\HH$ which has at least one zero and $\cC:\DD\to\HH$ be a linear bounded and injective operator such that $\cC\cC^{*}$ is an 
  admissible preconditioner of $\cA$. Then the sequence $\seq{w^{k}}$ generated by iteration~\eqref{eq:reducedHPE-inclusion}-\eqref{eq:reducedHPE-update} with $\inf \lambda_{k}>0$ converges weakly some $w^{*}$ which is a zero of $\cC^{*}\parcomp \cA$.

  Moreover it holds that $u^{*}= (\cC\cC^{*}+\cA)^{-1}\cC w^{*}$ is a zero of $\cA$ and if $(\cC\cC^{*}+\cA)^{-1}$ is Lipschitz continuous, then $\seq{\tu^{k}}$ converges weakly to $u^{*}$.
\end{theorem}
\begin{proof}
  We denote $\cM = \cC\cC^{*}$ and define the sequences $\seq{u^{k}}, \seq{v^{k}}$ with $\cC^{*}u^{k} = w^{k}$, $\cC^{*}v^{k}=z^{k}$. The iteration~\eqref{eq:reducedHPE-inclusion}-\eqref{eq:reducedHPE-update} is similar to the unreduced iteration~\eqref{eq:HPE-inclusion}-\eqref{eq:HPE-update} since there one accesses the variables $u^{k}$ and $v^{k}$ only via $\cC^{*}u^{k}$ and $\cC^{*}v^{k}$, respectively. Proposition~\ref{prop:fundamental-estimates-HPE} shows that the sequence $\seq{w^{k} = \cC^{*}u^{k}}$ is bounded and hence we have a weak cluster point $w^{*}$ and a subsequence $\seq{w^{k_{j}}}$ weakly converging to it. From~\eqref{eq:reducedHPE-inclusion} we have $\cC z^{k_{j}}\in \cA \tu^{k_{j}}$ which implies $\cC^{*}\tu^{k_{j}}\in \cC^{*}\cA^{-1}\cC z^{k_{j}}$ and this means that 
\begin{align*}
z^{k_{j}}\in (\cC^{*}\cA^{-1}\cC)^{-1}\cC^{*}\tu^{k_{j}} = (\cC^{*}\parcomp \cA)\cC^{*}\tu^{k_{j}}.
\end{align*}
From Proposition~\ref{prop:fundamental-estimates-HPE} we also get that 
\begin{align*}
\norm{\lambda_{k}z^{k}} = \norm{\lambda_{k}\cC^{*}v^{k}}\to 0\quad\text{and}\quad \norm{\cC^{*}u^{k_{j}}-w^{k_{j}}} = \norm{\cC^{*}u^{k_{j}}-\cC^{*}\tu^{k_{j}}}\to 0.
\end{align*}
since $\inf \lambda_{k}>0$ we get that $z^{k}\to 0$ and we also have $\cC^{*}\tu^{k_{j}}\wto w^{*}$.
Now Theorem 2.13 in~\cite{BCLN22} shows that $\cC^{*}\parcomp \cA$ is maximal monotone if $\cA$ is monotone and especially it has weak-strong closed graph. Hence from $j\to\infty$ we conclude $0\in (\cC^{*}\parcomp \cA)w^{*}$.

Let us now show that the full sequence $\seq{w^{k}}$ converges weakly: From Proposition~\ref{prop:fundamental-estimates-HPE} we get that $\norm{w^{k}-w}$ is decreasing for any zero of $\cC^{*} \parcomp \cA$ and hence, has a limit which we call $\ell(w)$. Now consider two subsequences $\seq{w^{k_{j}}}$ and $\seq{w^{l_{j}}}$ with weak limits $w^{*}$ and $w^{**}$, respectively. If we assume $w^{*}\neq w^{**}$ we get from Opial's lemma
\begin{align*}
  \liminf_{j\to\infty}\norm{w^{k_{j}}-w^{*}}_{\cM}&<\liminf_{j\to\infty}\norm{w^{k_{j}}-w^{**}}_{\cM}\\
  \liminf_{j\to\infty}\norm{w^{l_{j}}-w^{**}}_{\cM}&<\liminf_{j\to\infty}\norm{w^{l_{j}}-w^{*}}_{\cM}
\end{align*}
which implies $\ell(w^{*})<\ell(w^{**})$ and $\ell(w^{**})<\ell(w^{*})$ which is a contradiction. Thus $w^{*}=w^{**}$, i.e. we only have one cluster point of the full sequence.

Now assume that $(\cC\cC^{*} + \cA)^{-1}$ is Lipschitz. Theorem~\ref{thm:weak-conv-HPE} shows that $\seq{\tu^{k}}$ converges weakly to some $u^{**}$ which is a zero of $\cA$. Thus, since $C$ is linear and bounded we get that $\seq{\cC^{*}\tu^{k}}$ converges weakly to $\cC^{*}u^{**}$. From Proposition~\ref{prop:fundamental-estimates-HPE} (i) we get 
\begin{align*}
\norm{\cC^{*}\tu^{k+1}-w^{k}} = \norm{\cC^{*}\tu^{k+1}-\cC^{*}u^{k}} = \norm{\tu^{k+1}-u^{k}}_{\cM}\leq \tfrac1{1-\sigma}\norm{\lambda_{k+1}v^{k+1}}_{\cM}
\end{align*}
and the right hand side converges to zero. Hence, the weak limit of $\seq{\cC^{*}\tu^{k+1}}$ has to be $w^{*}$. Hence, we conclude $\cC^{*}u^{**} = w^{*}$ and 
\begin{align*}
u^{**} = (\cM+\cA)^{-1}\cM u^{**} = (\cM+\cA)^{-1}\cC\cC^{*}u^{**} = (\cM+\cA)^{-1}\cC w^{*} = u^{*}.
\end{align*}
\end{proof}

\begin{remark} The theoretical guarantees provided in our work are in line with the related literature on proximal-point type methods as well as on degenerate preconditioned methods (see, e.g., \cite{solodov1999,solodov2000, alves2020, eckstein2018, eckstein2013, BCLN22}). Convergence rates for sequences generated by proximal point type methods can be proved, in general, under stronger assumptions on the maximal monotone operator as, for instance, strong monotonicity \cite{AS2016, alves2016}. Convergence rates for the residuals (in terms of enlargements) were obtained in \cite{monteiro2010}. One could follow the latter work of Monteiro and Svaiter and try to prove convergence rates for the residuals of the outer iteration, but this is beyond the scope of the present paper. On the other hand, it is not clear how to prove rates when we consider outer and inner iterations, since we can't predict how many inner iterations will be required to completed each inner loop. The situation was somehow improved for a relative-error version of the
Douglas-Rachford method in \cite{AlvesGeremia} where a tolerance was prescribed for each inner loop (see equation (25) in Algorithm 3 of \cite{AlvesGeremia}). Instead of following this idea and introducing a different method we wanted to keep the analysis and exposition as simple as possible.
\end{remark}

\section{Applications to splitting methods}
\label{sec:applications}

In this section we consider the problem of finding a zero in the sum of maximal monotone operators. The aim of the section is to show that the degenerate HPE framework can be applied in the context of splitting methods, deriving new inexact schemes together with convergence results. We start by recovering the Eckstein-Yao method, which is an inexact version of the celebrated Douglas-Rachford method, continue with a new inexact version of the Chambolle-Pock method and conclude the section with a new inexact version of the Davis-Yin algorithm. We also comment about the fact that, thanks to the general framework we developed, many more methods can be included.

\subsection{Inexact Douglas-Rachford: retrieving Eckstein-Yao algorithm using the degenerate HPE framework}
\label{sec:DR}

In this paragraph we make use of the degenerate HPE framework to retrieve the inexact version of the Douglas-Rachford method introduced by Eckstein and Yao in \cite{eckstein2018}, later extended in \cite{SvaiterInexactDR} and further analyzed in \cite{AlvesGeremia}. The method is suited to find solutions for problems of the form
\begin{equation}\label{eq:2op_problem}
\text{find } x \in H \quad \text{such that } \quad 0\in A_1 x + A_2 x,
\end{equation}
where $A_1,A_2$ are maximal monotone operators on the real Hilbert space $H$. Let introduce the operators
\begin{equation}\label{eq:A_block_DR}
	\mathcal{A} := 
	\begin{bmatrix} \tau A_1 & I & -I\\
		-I &  \tau A_2 & I\\
		I & -I &  0
	\end{bmatrix}, \ \mathcal{M} :=
	\begin{bmatrix} I & -I & I\\
		-I & I & -I\\
		I & -I & I
	\end{bmatrix}, \ \mathcal{C} :=
\begin{bmatrix}
I \\
-I\\
I
\end{bmatrix},
\end{equation}
and consider $\HH:=H^3$. It was shown in \cite{BCLN22} and \cite{bredies2022graph} that the problem of finding $u=(x_1,x_2,v)\in \HH$ such that $0\in \cA u$ is equivalent to find a solution of \eqref{eq:2op_problem}. In particular, if $u=(x_1,x_2,v)\in \HH$ satisfies $0\in \cA u$, then $x_1=x_2$ and $x=x_1=x_2$ solves \eqref{eq:2op_problem}. On the other hand, if $x\in H$ solves \eqref{eq:2op_problem} then there exists $v\in H$ such that $u=(x,x,v)$ satisfies $0\in\cA u$.

Applying the reduced HPE method (Algorithm \ref{alg:HPE_r}) using the operator $\cA$, the preconditioner $\cM=\cC\cC^*$ and $\lambda_k \equiv 1$, we obtain Algorithm \ref{alg:HPE_EY}. In fact, Step 4 of the algorithm guarantees that
\begin{equation*}
\widetilde{x}_1^{k+1}-\tau a_1^{k+1}\in\widetilde{x}_2^{k+1}+\tau A_2 \widetilde{x}_2^{k+1},
\end{equation*}
and so
\begin{equation}\label{eq:HPE_EY_1}
	0\in -\widetilde{x}_1^{k+1}+\widetilde{x}_2^{k+1}+\tau A_2 \widetilde{x}_2^{k+1}+\tau a_1^{k+1}.
\end{equation}
Defining now $\widetilde{q}^{k+1}:=\tau a_1^{k+1} +2 \widetilde{x}_2^{k+1}-\widetilde{x}_1^{k+1}$, we have
\begin{equation}\label{eq:HPE_EY_2}
\widetilde{x}_1^{k+1}-\widetilde{x}_2^{k+1} =  \tau a_1^{k+1} + \widetilde{x}_2^{k+1}-\widetilde{q}^{k+1},
\end{equation}
and subtracting \eqref{eq:HPE_EY_2} from \eqref{eq:HPE_EY_1} we obtain
\begin{equation}\label{eq:HPE_EY_3}
\widetilde{x}_2^{k+1}-\widetilde{x}_1^{k+1}\in \tau A_2 \widetilde{x}_2^{k+1}-\widetilde{x}_1^{k+1} + \widetilde{q}^{k+1}.
\end{equation}
Combining \eqref{eq:HPE_EY_2} and \eqref{eq:HPE_EY_3} and introducing $z^{k+1}:=\widetilde{x}_1^{k+1}-\widetilde{x}_2^{k+1}$, we have found an element $\widetilde{u}^{k+1}=(\widetilde{x}_1^{k+1},\widetilde{x}_2^{k+1},\widetilde{q}^{k+1})\in H^3$ and an element $z^{k+1}\in H$, such that
\[\begin{pmatrix}
	z^{k+1}\\
	-z^{k+1}\\
	z^{k+1}
\end{pmatrix}\in 	\begin{bmatrix} \tau A_1 & I & -I\\
-I  &  \tau A_2 & I\\
I & -I &  0
\end{bmatrix} \begin{pmatrix}
	\widetilde{x}_1^{k+1}\\
	\widetilde{x}_2^{k+1}\\
	\widetilde{q}^{k+1}
\end{pmatrix},\]
which coincides with the inclusion $\cC z^{k+1}\in \cA \widetilde{u}^{k+1}$. The control in Step 5 of Algorithm \ref{alg:HPE_EY} guarantees the inequality in Step 4 of Algorithm \ref{alg:HPE_r}. In fact, the step give us
\begin{equation}\label{eq:HPE_EY_4}
\|\widetilde{x}_1^{k+1}+\tau a_1^{k+1}-w^{k}\|\leq \sigma \|\widetilde{x}_2^{k+1}+\tau a_1^{k+1}-w\|.
\end{equation}
From \eqref{eq:HPE_EY_2} and the definition of $z^{k+1}$, we have $\tau a_1^{k+1}+\widetilde{x}_2^{k+1}=\widetilde{q}^{k+1}+\widetilde{x}_1^{k+1}-\widetilde{x}_2^{k+1}$ and $z^{k+1}=\widetilde{x}_1^{k+1}-\widetilde{x}_2^{k+1}$, then
\[\|\widetilde{x}_1^{k+1}+\tau a_1^{k+1}-w^{k}\|=\|z^{k+1}+\widetilde{x}_1^{k+1}-\widetilde{x}_2^{k+1}+\widetilde{q}^{k+1}-w^k\|=\|z^{k+1}+\cC^* \widetilde{u}^{k+1}-w^{k}\|\]
and
\[\|\widetilde{x}_2^{k+1}+\tau a_1^{k+1}-w^k\|=\|\widetilde{x}_1^{k+1}-\widetilde{x}_2^{k+1}+\widetilde{q}^{k+1}-w^k\|=\|\cC^*\widetilde{u}^{k+1}-w^k\|.\]
So that \eqref{eq:HPE_EY_4} coincides with the inequality in Step 4 of Algorithm \ref{alg:HPE_r}.
From the definition of $\widetilde{w}^k$ and $z^{k+1}$, the update in Step 9 of Algorithm \ref{alg:HPE_EY} coincides with the one in Step 7 of Algorithm \ref{alg:HPE_r}.

Including this method in the degenerate HPE framework allow us to immediately retrieve the convergence result in \cite{eckstein2018}.

\mgap

\begin{theorem}
    Let $A_1,A_2$ be maximal monotone operators such that problem \eqref{eq:2op_problem} admits a solution. Let $\seq{\widetilde{x}_1^{k}}$ and 
    $\seq{\widetilde{x}_2^{k}}$ be sequences generated by Algorithm \ref{alg:HPE_EY} (where we suppose that the inner loop always end in a finite number of iterations). Then $\seq{\widetilde{x}_1^{k}}$ and $\seq{\widetilde{x}_2^{k}}$ converge weakly to a solution of \eqref{eq:2op_problem}.
\end{theorem}
\begin{proof}
    Let $\cA$, $\cM$ and $\cC$ as in \eqref{eq:A_block_DR}. Then, as discussed in \cite[Section 3]{BCLN22}, the operator $\cA$ is maximal monotone, $\cM=\cC\cC^*$ is an admissible preconditioner for $\cA$ and $(\cM+\cA)^{-1}$ is Lipschitz continuous. Then, applying Theorem \ref{thm:convergence_reduced}, we obtain $(\widetilde{x}_1^{k}, \widetilde{x}_2^{k}, \widetilde{q}^{k})=\widetilde{u}^k \wto u^*$, with $u^*$ solution to $0\in \cA u$. This implies the thesis.
\end{proof}

\begin{algorithm}[t]
\setstretch{1.1}
\begin{algorithmic}[1]
	\STATE \textbf{Initialize:} $w^0\in H$ and $\sigma \in [0,1)$
	\FOR{$k=1, 2 \dots$}
	\STATE Define $(\widetilde{x}_1^{k+1},a_1^{k+1})$ such that $a_1^{k+1}\in A_1 \widetilde{x}_1^{k+1}$ fulfill \[\tau a_1^{k+1}+\widetilde{x}_1^{k+1} \approx w^{k}.\]\vspace*{-\baselineskip}
	\STATE Find $(\widetilde{x}_2^{k+1},a_2^{k+1})$ such that $a_2^{k+1}\in A_2 \widetilde{x}_2^{k+1}$ and \[\tau a_2^{k+1}+\widetilde{x}_2^{k+1}=\widetilde{x}_1^{k+1}-\tau a_1^{k+1}.\] (We compute $\widetilde{x}_2^{k+1}=J_{\tau A_2}(\widetilde{x}_1^{k+1}-a_1^{k+1})$ and take\\ $\tau a_2^{k+1}=\widetilde{x}_1^{k+1}-\tau a_1^{k+1}-\widetilde{x}_2^{k+1}$.)
	\WHILE{$\|\widetilde{x}_1^{k+1}+\tau a_1^{k+1}-w^{k}\|> \sigma \|\widetilde{x}_2^{k+1}+\tau a_1^{k+1}-w^{k}\|$ }
        \STATE  Improve $(\widetilde{x}_1^{k+1},a_1^{k+1})$ with $a_1^{k+1}\in A_1 \widetilde{x}_1^{k+1}$ to reduce\\ $\norm{\widetilde{x}_1^{k+1} + \tau a_1^{k+1} - w^{k}}$.
	\STATE Similar to step 4 find $(\widetilde{x}_2^{k+1},a_2^{k+1})$ such that $a_2^{k+1}\in A_2 \widetilde{x}_2^{k+1}$ and \[\tau a_2^{k+1}+\widetilde{x}_2^{k+1}=\widetilde{x}_1^{k+1}-\tau a_1^{k+1}.\]
        \ENDWHILE
	\STATE Update
        \[
	w^{k+1}=w^{k}-\tau (a_1^{k+1}+a_2^{k+1}),\]
	(or, equivalently, $w^{k+1}=w^{k}+\widetilde{x}_2^{k+1}-\widetilde{x}_1^{k+1}$).
	\ENDFOR
\end{algorithmic}
	\caption{The method introduced by Eckstein and Yao in \cite{eckstein2018}.}\label{alg:HPE_EY}
\end{algorithm}

\subsection{Chambolle-Pock aka PDHG}
\label{sec:CP}

In this section we are interested in solving problems of the form
\begin{align}
  \label{eq:CP-inclusion}
  \text{find } x \in H \quad \text{such that } \quad 0\in A_1 (x) + K^* A_2 K x,
\end{align}
where $H_1,H_2$ are two real Hilbert spaces, $A_1:H_1 \to 2^{H_1}$, $A_2: H_2 \to 2^{H_2}$ are maximal monotone operators and $K: H_1 \to H_2$ is a linear bounded operator. The problem is equivalent to finding $(x,y)$ such that $0\in A_1x+ K^*y$ and $0 \in - K x + A_2^{-1}y$. For this reason, we consider the operators
\begin{equation}\label{eq:A_block_CP}
\cA = \begin{bmatrix}
    A_1 & K^*\\
    -K & A_2^{-1}
\end{bmatrix}, \quad \cM = \begin{bmatrix}
    \frac{1}{\tau}I & -K^*\\
    -K & \frac{1}{\theta} I
\end{bmatrix},\end{equation}
the space $\HH:=H_1\times H_2$ and the corresponding problem
\begin{equation}\label{eq:CP_KKT}
\text{find } u=(x,y) \in \HH \quad \text{such that } \quad 0\in \cA u.
\end{equation}
The (degenerate) preconditioned HPE method defined using the operator $\cA$, the preconditioner $\cM$ and $\lambda_k \equiv 1$, writes as Algorithm \ref{alg:HPE_CP}. In fact, using step 3 and 4 of the algorithm we find $\widetilde{u}^{k+1}=(\widetilde{x}^{k+1},\widetilde{y}^{k+1})$ and if we define
\[\begin{aligned}
    v_1^{k+1}&:=\tau a_1^{k+1}+\tau K^{*}y^k\\
    v_2^{k+1}&:= y^k-\widetilde{y}^{k+1},
\end{aligned}\]
the variable $v^{k+1}=(v_1^{k+1},v_2^{k+1})$ satisfies $\cM v^{k+1}\in \cA \widetilde{u}^{k+1}$.
In fact, 
\[\begin{aligned}
    \begin{bmatrix}
    \frac{1}{\tau}I & -K^*\\
    -K & \frac{1}{\theta} I
\end{bmatrix}\begin{pmatrix}
    v_1^{k+1}\\
    v_2^{k+1}
\end{pmatrix}&=\begin{pmatrix}
    a_1^{k+1}+K^{*}y^k-K^{*}(y^k-\widetilde{y}^{k+1})\\
    -K(\tau a_1^{k+1}+\tau K^*y^k)+\theta^{-1}(y^k-\widetilde{y}^{k+1})
\end{pmatrix}\\
    &=\begin{pmatrix}
    a_1^{k+1}+K^*\widetilde{y}^{k+1}\\
    -K\widetilde{x}^{k+1}+K(\widetilde{x}^{k+1}-\tau a_1^{k+1}+\tau K^*y^k)+\theta^{-1}(y^k-\widetilde{y}^{k+1})
\end{pmatrix}\\
& \in \begin{pmatrix}
    A_1 \widetilde{x}^{k+1}+K^* \widetilde{y}^{k+1}\\
    -K\widetilde{x}^{k+1}+A_2^{-1}\widetilde{y}^{k+1}
\end{pmatrix}\\
& = \begin{bmatrix}
    A_1 & K^*\\
    -K & A_2^{-1}
\end{bmatrix} \begin{pmatrix}
    \widetilde{x}^{k+1}\\
    \widetilde{y}^{k+1}
\end{pmatrix} = \cA \begin{pmatrix}
    \widetilde{x}^{k+1}\\
    \widetilde{y}^{k+1}
\end{pmatrix}
\end{aligned}
\]

\begin{algorithm}[t]
\setstretch{1.1}
\begin{algorithmic}[1]
	\STATE \textbf{Initialize:} $x^0\in H_1, y^0\in H_2$, $\sigma \in [0,1)$, $\tau \theta \|K\|^2 \leq 1$
	\FOR{$k=1, 2 \dots$}
	\STATE Define $(\widetilde{x}^{k+1},a^{k+1})$ such that $a^{k+1}\in A_1 \widetilde{x}^{k+1}$ fulfill \[\tau a^{k+1}+\widetilde{x}^{k+1} \approx x^{k}-\tau K^*y^k.\]\vspace*{-\baselineskip}
	\STATE Compute $\widetilde{y}^{k+1}=J_{\theta A_2^{-1}}\left(y^k+\theta K\left(\widetilde{x}^{k+1}-\tau (a^{k+1}+K^{*}y^k)\right)\right)$.
	\WHILE{\[\tau^{-1}\|\tau a^{k+1}+\widetilde{x}^{k+1}-(x^k-\tau K^*y^k)\|^2> \sigma^2 \|(\widetilde{x}^{k+1}-x^k,\widetilde{y}^{k+1}-y^k)\|_{\cM}^2\] }
        \STATE Improve $(\widetilde{x}^{k+1},a^{k+1})$ with $a^{k+1}\in A_1 \widetilde{x}^{k+1}$ to reduce\\ $\|\tau a^{k+1}+\widetilde{x}^{k+1}-(x^k-\tau K^*y^k)\|$.
	\STATE Compute $\widetilde{y}^{k+1}=J_{\theta A_2^{-1}}\left(y^k+\theta K\left(\widetilde{x}^{k+1}-\tau (a^{k+1}+K^{*}y^k)\right)\right)$.
        \ENDWHILE
	\STATE Update
        \[\begin{aligned}
        x^{k+1}&=(x^k-\tau K^*y^k)-\tau a^{k+1}\\
        y^{k+1}&=\widetilde{y}^{k+1}.
        \end{aligned}\]
	\ENDFOR
\end{algorithmic}
	\caption{Inexact Chambolle-Pock method.}\label{alg:HPE_CP}
\end{algorithm}

With this choice, it is easy to see that Step 5 of algorithm checks for the condition \eqref{eq:HPE-error-condition} and the update in Step 9 coincides with \eqref{eq:HPE-update}.

Since the introduced algorithm is an example of (degenerate) HPE method, this allows us to retrieve the following convergence result.

\mgap

\begin{theorem}
    Let $A_1:H_1\to 2^{H_1},A_2:H_2\to 2^{H_2}$ be maximal monotone operators and $K:H_1\to H_2$ a linear bounded operator such that problem \eqref{eq:CP-inclusion} admits a solution. Let 
    $\seq{\widetilde{x}^{k}}$ and 
    $\seq{\widetilde{y}^{k}}$ the sequences generated by Algorithm \ref{alg:HPE_EY} (where we suppose that the inner loop always end in a finite number of iterations). Then 
    $\seq{\widetilde{u}^{k}\coloneqq (\widetilde{x}^{k},\widetilde{y}^{k})}$ converges weakly to a solution $u$ of \eqref{eq:CP_KKT}. In particular, $\seq{\widetilde{x}^{k}}$ converges weakly to a solution of \eqref{eq:CP-inclusion}.
\end{theorem}
\begin{proof}
    Let $\cA$ and $\cM$ as in \eqref{eq:A_block_CP}. Then, as discussed in \cite[Section 3]{BCLN22}, the operator $\cA$ is maximal monotone, $\cM$ is an admissible preconditioner for $\cA$ and $(\cM+\cA)^{-1}$ is Lipschitz continuous. Then, applying Theorem \ref{thm:weak-conv-HPE}, we obtain the thesis.
\end{proof}
\mgap

\begin{remark}
    It is well-known that the Chambolle-Pock method is a generalization of the Douglas-Rachford algorithm. Thus, the new Algorithm \ref{alg:HPE_CP} is a generalization of the Eckstein-Yao method (Algorithm \ref{alg:HPE_EY}). In fact, we can recover of the Eckstein-Yao method when $K=I$ using $\theta=\tau = 1$ and the substitution (reduction of variables) $w^k=x^k+y^k$, for all $k\in \N$.
\end{remark}

\mgap

\begin{remark}
    We recall that the algorithm is well suited to solve problems of the form $\min_x f_1(x)+f_2(Kx)$ with $f_1 \in \Gamma_0(H_1)$, $f_2 \in \Gamma_0(H_2)$ and relative saddle point problems of the form $\min_x \sup_y f_1(x)-f_2^*(y)+\langle Kx, y \rangle$. In fact, the KKT conditions for this kind of problems read as
    \[0\in \begin{bmatrix}
    \partial f_1 & K^*\\
    -K & \partial f_2^*
\end{bmatrix} \begin{pmatrix}
    x\\
    y
\end{pmatrix}. \]
So that the problem is of the form \eqref{eq:CP_KKT} and a solution can be found using Algorithm \ref{alg:HPE_CP}.
\end{remark}

\subsection{Inexact Davis-Yin method}
\label{sec:DY}
Making use of the generality of the degenerate HPE framework, we introduce in this section a new inexact algorithm. The algorithm we introduce is an inexact version of the so-called Davis-Yin method \cite{davis2017} (also called three operator splitting or forward Douglas-Rachford \cite{Briceo2012,Raguet2013,Raguet2019}). An inexact version of the Davis-Yin algorithm is introduced in \cite{ZTC18}, but relies on summable error sequences while our method uses a relative-error criterion. The method is suited to find solution for problems of the form 
\begin{equation}\label{eq:3op_problem}
\text{find } x \in H \quad \text{such that } \quad 0\in A_1x + A_2x + Bx,
\end{equation}
where $A_1,A_2,B$ are maximal monotone operators and $B$ is $1/\beta$-cocoercive. In order to introduce the algorithm, we follow \cite{BCLN22} and introduce similar operators to the ones in \cite[Section 3]{BCLN22}. Let
\begin{equation}\label{eq:A_block_DY}
	\mathcal{A} := 
	\begin{bmatrix} \alpha I + \tau A_1 & I & -I\\
		-I-2\alpha I + \tau B & \alpha I + \tau A_2 & I\\
		I & -I &  0
	\end{bmatrix}, \ \mathcal{M} :=
	\begin{bmatrix} I & -I & I\\
		-I & I & -I\\
		I & -I & I
	\end{bmatrix}, \ \mathcal{C} :=
\begin{bmatrix}
I \\
-I\\
I
\end{bmatrix}.
\end{equation}
Defining the parameter $\gamma=\frac{\tau}{1+\alpha}$ and $\widetilde{w}^{k}=\frac{w^k}{1+\alpha}$, for all $k\in \N$, we have that the reduced HPE method (Algorithm \ref{alg:HPE_r}) defined using the operator $\cA$, the preconditioner $\cM=\cC\cC^*$ and the choice $\lambda_k \equiv 1$, writes as Algorithm \ref{alg:HPE_DY}. 

\begin{algorithm}[t]
\setstretch{1.1}
	\begin{algorithmic}[1]
		\STATE \textbf{Initialize:} $w^0\in H$ and $\sigma \in [0,1)$, $\gamma \in \left(0,\frac{2}{\beta}\right)$, $\alpha = \frac{\gamma \beta}{4-\gamma \beta}$.
		\FOR{$k=1, 2 \dots$}
		\STATE First define $(\widetilde{x}_1^{k+1},a_1^{k+1})$ such that $a_1^{k+1}\in A_1 \widetilde{x}_1^{k+1}$ and \[\gamma a_1^{k+1}+\widetilde{x}_1^{k+1} \approx \widetilde{w}^{k}\]\vspace*{-\baselineskip}.
		\STATE Then take
		\[\widetilde{x}_2^{k+1}=J_{\gamma A_2}\left(\widetilde{x}_1^{k+1}-\gamma a_1^{k+1}-\gamma B\widetilde{x}_1^{k+1}\right).\]
		\WHILE{$\|\widetilde{x}_1^{k+1}+\gamma a_1^{k+1}-\widetilde{w}^{k}\| > \sigma\left\|\frac{\alpha \widetilde{x}_1^{k+1} + \widetilde{x}_2^{k+1}}{1+\alpha} + \gamma a_1^{k+1} - \widetilde{w}^{k}\right\|$}
		\STATE Improve $(\widetilde{x}_1^{k+1},a_1^{k+1})$ with $a_1^{k+1}\in A_1 \widetilde{x}_1^{k+1}$ to reduce\\ $\norm{\widetilde{x}_1^{k+1} + \gamma a_1^{k+1} - \widetilde{w}^{k}}$.
                \STATE Then take
		\[\widetilde{x}_2^{k+1}=J_{\gamma A_2}\left(\widetilde{x}_1^{k+1}-\gamma a_1^{k+1}-\gamma B\widetilde{x}_1^{k+1}\right).\]
                \ENDWHILE
		\STATE Update
                \[
		\widetilde{w}^{k+1}=\widetilde{w}^{k}+\frac{1}{1+\alpha}(\widetilde{x}_2^{k+1}-\widetilde{x}_1^{k+1}).\]
		\ENDFOR
	\end{algorithmic}
	\caption{Inexact version of the Davis-Yin algorithm.}\label{alg:HPE_DY}
\end{algorithm}

In fact, Step 4 guarantees that
\begin{equation*}
(1+\alpha)\widetilde{x}_1^{k+1}-(1+\alpha)\gamma a_1^{k+1}-(1+\alpha)\gamma B \widetilde{x}_1^{k+1}\in(1+\alpha)\widetilde{x}_2^{k+1}+(1+\alpha)\gamma A_2 \widetilde{x}_2^{k+1}.
\end{equation*}
Using the fact that $\tau=(1+\alpha)\gamma$ and rearranging, we obtain
\begin{equation}\label{eq:HPE_DY_1}
	0\in-(1+\alpha)\widetilde{x}_1^{k+1}(1+\alpha)\widetilde{x}_2^{k+1}+\tau A_2 \widetilde{x}_2^{k+1}+\tau a_1^{k+1}+\tau B \widetilde{x}_1^{k+1}-(1+\alpha)\widetilde{x}_1^{k+1}.
\end{equation}
Defining now $\widetilde{q}^{k+1}:=\alpha \widetilde{x}_1^{k+1} +\tau a_1^{k+1} +2 \widetilde{x}_2^{k+1}-\widetilde{x}_1^{k+1}$, we have
\begin{equation}\label{eq:HPE_DY_2}
\widetilde{x}_1^{k+1}-\widetilde{x}_2^{k+1} = \alpha \widetilde{x}_1^{k+1} +\tau a_1^{k+1} + \widetilde{x}_2^{k+1}-\widetilde{q}^{k+1},
\end{equation}
and subtracting \eqref{eq:HPE_DY_2} from \eqref{eq:HPE_DY_1} we obtain
\begin{equation}\label{eq:HPE_DY_3}
\widetilde{x}_2^{k+1}-\widetilde{x}_1^{k+1}\in \alpha \widetilde{x}_2^{k+1} + \tau A_2 \widetilde{x}_2^{k+1}-\widetilde{x}_1^{k+1}-2\alpha \widetilde{x}_1^{k+1} +\tau B \widetilde{x}_1^{k+1} + \widetilde{q}^{k+1}.
\end{equation}
Combining \eqref{eq:HPE_DY_2} and \eqref{eq:HPE_DY_3} and introducing $z^{k+1}:=\widetilde{x}_1^{k+1}-\widetilde{x}_2^{k+1}$, we have found an element $\widetilde{u}^{k+1}=(\widetilde{x}_1^{k+1},\widetilde{x}_2^{k+1},\widetilde{q}^{k+1})\in H^3$ and an element $z^{k+1}\in H$, such that
\[\begin{pmatrix}
	z^{k+1}\\
	-z^{k+1}\\
	z^{k+1}
\end{pmatrix}\in 	\begin{bmatrix} \alpha I + \tau A_1 & I & -I\\
-I-2\alpha I + \tau B & \alpha I + \tau A_2 & I\\
I & -I &  0
\end{bmatrix} \begin{pmatrix}
	\widetilde{x}_1^{k+1}\\
	\widetilde{x}_2^{k+1}\\
	\widetilde{q}^{k+1}
\end{pmatrix},\]
which is exactly $\cC z^{k+1}\in \cA \widetilde{u}^{k+1}$. The control in Step 5 of Algorithm \ref{alg:HPE_DY} guarantees the inequality in Step 4 of Algorithm \ref{alg:HPE_r}. In fact, the step give us
\begin{equation}\label{eq:HPE_DY_4}
\|(1+\alpha)\widetilde{x}_1^{k+1}+\tau a_1^{k+1}-w^{k}\|\leq \sigma \|\alpha\widetilde{x}_1^{k+1}+\widetilde{x}_2^{k+1}+\tau a_1^{k+1}-w\|.
\end{equation}
From \eqref{eq:HPE_DY_2} and the definition of $z^{k+1}$, we have $\tau a_1^{k+1}+\alpha\widetilde{x}_1^{k+1}+\widetilde{x}_2^{k+1}=\widetilde{q}^{k+1}+\widetilde{x}_1^{k+1}-\widetilde{x}_2^{k+1}$ and $z^{k+1}=\widetilde{x}_1^{k+1}-\widetilde{x}_2^{k+1}$, then
\begin{align*}
    \|(1+\alpha)\widetilde{x}_1^{k+1}+\tau a_1^{k+1}-w^{k}\| & =\|z^{k+1}+\widetilde{x}_1^{k+1}-\widetilde{x}_2^{k+1}+\widetilde{q}^{k+1}-w^k\|\\ & =\|z^{k+1}+\cC^* \widetilde{u}^{k+1}-w^{k}\|
    \end{align*}
and
\[\|\alpha\widetilde{x}_1^{k+1}+\widetilde{x}_2^{k+1}+\tau a_1^{k+1}-w^k\|=\|\widetilde{x}_1^{k+1}-\widetilde{x}_2^{k+1}+\widetilde{q}^{k+1}-w^k\|=\|\cC^*\widetilde{u}^{k+1}-w^k\|.\]
So that \eqref{eq:HPE_DY_4} coincides with the inequality in Step 4 of \eqref{alg:HPE_r}.
From the definition of $\widetilde{w}^k$ and $z^{k+1}$, the update in Step 9 of Algorithm \ref{alg:HPE_DY} coincides with the one in Step 7 of Algorithm \ref{alg:HPE_r}.

Since the introduced algorithm is an example of degenerate HPE method, this allows us to retrieve the following convergence result.

\mgap

\begin{theorem}
    Let $A_1,A_2,B$ be maximal monotone operators such that $B$ is $1/\beta$-cocoercive and problem \eqref{eq:3op_problem} admits a solution. Let 
    $\seq{\widetilde{x}_1^{k}}$ and 
    $\seq{\widetilde{x}_2^{k}}$ sequences generated by Algorithm \ref{alg:HPE_DY} (where we suppose that the inner loop always end in a finite number of iterations). Then $\seq{\widetilde{x}_1^{k}}$ and $\seq{\widetilde{x}_2^{k}}$ converge weakly to a solution of \eqref{eq:3op_problem}.
\end{theorem}
\begin{proof}
    Let $\cA$, $\cM$ and $\cC$ as in \eqref{eq:A_block_DY}. Then, as discussed in \cite[Section 3]{BCLN22}, the operator $\cA$ is maximal monotone, $\cM=\cC\cC^*$ is an admissible preconditioner for $\cA$ and $(\cM+\cA)^{-1}$ is Lipschitz continuous. Then, applying Theorem \ref{thm:convergence_reduced}, we obtain $(\widetilde{x}_1^{k}, \widetilde{x}_2^{k}, \widetilde{q}^{k})=\widetilde{u}^k \wto u^*$, with $u^*$ solution to $0\in \cA u$. This implies the thesis.
\end{proof}

\begin{remark}
    The Davis-Yin method comes as a generalization of the Douglas-Rachford algorithm that makes use of an additional forward term. Thus, the new Algorithm \ref{alg:HPE_CP} is a generalization of the Eckstein-Yao method (Algorithm \ref{alg:HPE_EY}), which we recover when $B=0$.
\end{remark}

\mgap

\begin{remark}\label{rem:DY-functions}
    We recall that the algorithm is well suited to solve problems of the form \[\min_x f_1(x)+f_2(x)+g(x),\] with $f_1,f_2,g \in \Gamma_0(H)$ and $g$ with Lipschitz gradient. In fact, under some mild qualification conditions \cite[Theorem 16.47]{BCombettes}, it is possible to recover a problem of the form \eqref{eq:3op_problem} by considering $A_1 = \partial f_1$, $A_2 = \partial f_2$ and $B = \nabla g$.
\end{remark}

\mgap

\begin{remark}
    As already mentioned in the introduction, the degenerate proximal point framework has been shown to include many existing methods. We presented in this section the inexact version of only a few of them. Other examples are: the splitting method introduced by Ryu in \cite{Ryu}, the parallel Douglas-Rachford introduced and analyzed in \cite{RyuYin_parallel, condat, Campoy2022}, the method introduced by Malitsky and Tam in \cite{MalitskyTam2023resolvent} and the sequential method introduced in \cite{bredies2022graph}. Through the framework introduced in the present paper, their inexact counterpart can be easily analyzed. We didn't present here the inexact versions of such algorithms for sake of exposition, however, some generalizations are explicitly reported in \cite{naldi2024thesis}. In a recent paper \cite{Artacho2024}, Aragón-Artacho, Campoy and López-Pastor introduced new graph-based forward-backward methods which are connected to some graph-based extensions of the Davis-Yin method studied in \cite{naldi2024thesis}. All of these methods are based on the degenerate proximal point framework proposed in \cite{BCLN22}. For this reason, the present work also paves the way for possible new inexact generalizations of these algorithms.
\end{remark}

\section{Experiments}

In this section we perform numerical experiments with the goal of illustrating the practical benefits of using approximate resolvents with relative errors. Even though the methodology we developed in this paper is applicable to a wide variety of splitting methods, we will focus our experiments on two cases: The Chambolle-Pock method as described in Section~\ref{sec:CP} and the Davis-yin method from Section~\ref{sec:DY}.

\subsection{Chambolle-Pock}
\label{sec:cp-experiment}

To illustrate the effectiveness of the HPE-version of the Chambolle-Pock method from Section~\ref{sec:CP} we consider the variational problem 
\begin{align}\label{eq:problem-CP}
\min_{x\in\R^{n}} \tfrac12\norm{Hx-f}_{2}^{2} + \lambda\norm{Dx}_{1}
\end{align}
for $H\in\R^{m\times n}$, $f\in \R^{m}$, $\lambda>0$ and $D\in\R^{k\times n}$.
This problem is prototypical for several variational approaches to inverse problems and imaging, e.g. when $H$ is a forward operator, $f$ is given data and the inverse problem is to approximately solve $Hx=f$ under the assumption that $f$ is perturbed by noise and $H$ is ill-conditioned. The parameter $\lambda$ is the regularization parameter and the term $\norm{Dx}_{1}$ is a regularization term that enforces that $Dx$ will be sparse. The famous (isotropic) total variation regularization is a special case of this when $D$ is a matrix of first order finite differences~\cite{rudin1992nonlinear,bredies2018mathematical}.

There are numerous ways to solve this problem with splitting methods. Especially the term $\tfrac12\norm{Hx-f}_2^2$ can be treated in different ways:
\begin{itemize}
\item Doing implicit steps, i.e. evaluating the respective resolvent of $\tfrac12\norm{Hf-f}_{2}^{2}$ exactly. This amount to solving a linear system with the operator $I + \tau H^{T}H$. We will do this in what we call \emph{Chambolle-Pock with implicit steps} or \emph{implicit Chambolle-Pock} and explain the details below.
\item Doing approximate implicit steps by using the method of conjugate gradients (CG) for the solution of the linear systems within the HPE framework as described in Algorithm~\ref{alg:HPE_CP}.
\item Doing forward steps: Since the term is $L$-smooth, one can do forward steps of the form $H^{T}(Hx-f)$. This is, for example, possible within the Condat-Vu method~\cite{condat2013primal,vu2013splitting}.
\item One use a different splitting and  also ``dualize over $H$'' (and $D$ at the same time) and then apply the Chambolle-Pock method. This also avoids solutions of linear systems and we call this method \emph{Chambolle-Pock with explicit steps} or \emph{explicit Chambolle-Pock} and describe the details below.
\end{itemize}

\subsubsection{Chambolle-Pock with implicit steps}
\label{sec:implicit-CP}

We can write~\eqref{eq:problem-CP} in standard form 
\begin{align*}
  \min_{x\in\R^{n}} \tfrac12\norm{Hx-f}_{2}^{2} + \lambda\norm{Dx}_{1} & = \min_{x}F(x) + G(Kx)\\
  \text{with}\quad F(x) = \tfrac12\norm{Hx-f}_{2}^{2},\quad K& =D,\quad G(z) = \lambda\norm{z}_{1}.
\end{align*}
We put this in correspondence with~\eqref{eq:CP-inclusion} by setting $A_{1}x = H^{T}(Hx-f)$, $K = D$ and $A_{2}x =\lambda\Sign(x) = \lambda\partial\norm{\cdot}_{1}(x)$ where $\Sign(x)$ is the componentwise multivalued sign function.

The inverse of $A_{2}$ is the multivalued map with the components
\begin{align*}
(A_{2}^{-1}x)_{i} = \left\{ 
  \begin{array}{l@{\quad:\quad}l}
    \emptyset & \abs{x_{i}}>\lambda\\
    ]-\infty,0] & x_{i}=-\lambda\\
    \left\{ 0 \right\} & \abs{x_{i}}<\lambda\\
    {[0,\infty[} & x_{i}=\lambda
\end{array}\right.
\end{align*}
and its resolvent (i.e. the proximal map of $G^{*}(y) = I_{\abs{\cdot}_{\infty}\leq \lambda}(y)$) is the clipping (and independent of $\theta>0$)
\begin{align*}
J_{\theta A_{2}^{-1}}(x) = \min(\max(x,-\lambda),\lambda) =: \operatorname{clip}_{[-\lambda,\lambda]}(x)
\end{align*}
(all operations applied componentwise).
The resolvent of $A_{1}$ (i.e. the proximal map of $F$) is
\begin{align*}
J_{\tau A_{1}}(x) =  \prox_{\tau F}(x) = (I + \tau H^{T}H)^{-1}(x + \tau H^{T}f)
\end{align*}
and hence, involves solving an $n\times n$ linear system.
The full iteration is 
\begin{align}\label{eq:implicitCP-iteration}
  \begin{split}
    x^{k+1} & = (I+\tau H^{T}H)^{-1}(x^{k} - \tau (D^{T}y^{k} - H^{T}f))\\
    y^{k+1} & = \operatorname{clip}_{[-\lambda,\lambda]}(y^{k} + \theta D(2x^{k+1}-x^{k})) 
  \end{split}
\end{align}
The method needs $\tau\theta \leq \norm{D}^{-2}$ and since we can estimate $\norm{D}\leq 2$ we take $\tau = \tfrac1{2\kappa}$, $\theta = \tfrac{\kappa}{2}$, with a scale factor $\kappa>0$.

\subsubsection{HPE Chambolle-Pock}
\label{sec:HPE-CP}

To avoid the exact solution of the linear system in the Chambolle-Pock method, we can apply the HPE framework from Section~\ref{sec:CP}, i.e. Algorithm~\ref{alg:HPE_CP}. Instead of evaluating the resolvent $J_{\tau A_{1}}$ exactly
we apply the method of conjugate gradients to improve the pair $(\tilde{x}^{k+1},a^{k+1})$. Note that for a given candidate $\tilde{x}^{k+1}$ there is a unique $a^{k+1} = H^{T}(H\tilde{x}^{k+1}-f)\in A_{1}\tilde{x}^{k+1}$.

For the HPE CP method we also need $\tau\theta \leq \norm{D}^{-2}\leq \tfrac14$ and we take again $\tau = \tfrac1{2\kappa}$, $\theta = \tfrac{\kappa}{2}$, with a scale factor $\kappa>0$.

\subsubsection{Condat-Vu}
\label{sec:condat-vu}

Another way to avoid the solution of the linear system is the Condat-Vu method. This method can be applied to problems of the form 
\begin{align*}
\min_{x\in \R^n} F(x) + \Phi(x) + G(Kx)
\end{align*}
where $H$ is an $L$-smooth convex function, and $F$ and $G$ are convex with simple proximal maps. We rewrite~\eqref{eq:problem-CP} in this form with 
\begin{align*}
F(x) = 0,\quad \Phi(x) = \tfrac12\norm{Hx-f}_{2}^{2},\quad G(z) = \lambda\norm{z}_{1},\quad K=D.
\end{align*}
The function $\Phi$ is indeed $L$-smooth with $L=\norm{H}^{2}$, and we can apply the Condat-Vu method~\cite{condat2013primal,vu2013splitting} which reads in this case as 
\begin{align}\label{eq:CV-iteration}
  \begin{split}
    x^{k+1} & = x^{k} - \tau (H^{T}(Hx^{k}-f) - D^{T}y^{k})\\
    y^{k+1} & = \prox_{I_{\norm{\cdot}_{\infty}\leq\lambda}}(y^{k} + \theta D(2x^{k+1}-x^{k})) 
  \end{split}
\end{align}
and converges as soon as the stepsizes $\tau,\sigma$ fulfill 
\begin{align*}
0<\tau<\tfrac2{\norm{H}^{2}},\quad 0<\theta < \left( \tfrac1\tau - \tfrac{\norm{H}^{2}}2 \right)\tfrac{1}{\norm{D}^{2}},
\end{align*}
cf.~\cite{condat2013primal,vu2013splitting}.

\subsubsection{Chambolle-Pock with explicit steps}
\label{sec:standard-CP}

Yet another way to avoid the solution of the linear system, we can also apply the plain Chambolle-Pock method after dualizing over both linear operators $H$ and $D$, i.e. we rewrite~\eqref{eq:problem-CP} as
\begin{align*}
  \min_{u\in\R^{n}} \tfrac12\norm{Hx-f}_{2}^{2} + \lambda\norm{Dx}_{1} & = \min_{x} \tilde{F}(x) + \tilde{G}(\tilde{K}x)\\
  \text{with}\quad\tilde{F}(x) = 0,\quad \tilde{K}x &=
  \begin{bmatrix}
    Hx\\Dx
  \end{bmatrix},\quad \tilde{G}(a,b) = \tfrac12\norm{a-f}_{2}^{2} + \lambda\norm{b}_{1}.
\end{align*}
This leads to the method
\begin{align}\label{eq:CP-iteration}
  \begin{split}
    x^{k+1} & = x^{k} - \tau(H^{T}u + D^{T}v)\\
    u^{k+1} & = \tfrac{1}{1+\theta}(u^{k} + \theta(H(2x^{k+1}-x^{k} -f)\\
    v^{k+1} & = \prox_{I_{\norm{\cdot}_{\infty}\leq \lambda}}(v^{k} + \theta D(2x^{k+1}-x^{k}))
  \end{split}
\end{align}
and convergence is guaranteed if we have $\theta\tau\leq \norm{K}^{2}$ (and we estimate $\norm{K}^{2}\leq\norm{H}^{2} + \norm{D}^{2}$). Here we take $\tau = \tfrac{1}{\kappa\norm{K}}$, $\theta = \tfrac{\kappa}{\norm{K}}$, for some $\kappa>0$.

\subsubsection{Experimental results}
\label{sec:CP-results}

In a first experiment we used $m=n=2000$ and generated an ill-conditioned matrix $H = U\Sigma V^{T}$ with two random orthonormal matrices $U,V$ and a diagonal matrix $\Sigma$ with diagonal entries $\sigma_{i} = \tfrac12 + \tfrac12\cos\left(\pi\tfrac{i-1}{m-1}\right)$, $i=1\dots,m$. In this way the singular values of $H$ decay from $1$ to $0$ along a cosine curve, so there is no big cluster of singular values. The resulting condition number was about $4.7\cdot10^{8}$. As matrix $D$ we take the $n-1\times n$ matrix that contains the first finite differences. As $f$ we take $f= Hx^{\dag} + \eta$ where $x^{\dag}$ is a piecewise constant and sparse signal, while $\eta$ is Gaussian noise.

We compare the following algorithms:
\begin{itemize}
\item implicit Chambolle-Pock (implicit CP) from Section~\ref{sec:implicit-CP},
\item inexact Chambolle-Pock with the HPE framework (HPE CP) from Algorithm~\ref{alg:HPE_CP} as described in Section~\ref{sec:HPE-CP},
\item the Condat-Vu method (CV) from Section~\ref{sec:condat-vu}, and
\item explicit Chambolle-Pock (explicit CP) from Section~\ref{sec:standard-CP}.
\end{itemize}

For the implicit Chambolle-Pock we also used the CG method to solve the linear system. However, we used a fixed relative tolerance of $10^{-8}$ for the residual to terminate the method (as it is usual for CG). In all cases where the CG method has been used it has been initialized with the current iterate $x$ as warmstart.

We made two runs with the following setups:
\begin{itemize}
\item First run: A large regularization of $\lambda=20$ and for HPE CP we demanded high accuracy by setting $\sigma=0.01$. The scaling of the stepsizes for HPE CP and implicit and explicit CP were tuned by hand to give best possible performance (we used $\kappa = 0.5$ for all methods).
\item Second run: A smaller regularization of $\lambda=1$ and a loser requirement for the accuracy of HPE CP by setting $\sigma=0.95$. The scaling of the stepsizes for HPE CP and implicit and explicit CP were tuned by hand to give best possible performance (we used $\kappa = 0.1$ for all methods).
\end{itemize}

We report the distance to the minimal objective value (computed with a larger number of iterations) in log-scale over iterations in Figures~\ref{fig:experiment1-CP-1} and~\ref{fig:experiment1-CP-2} and since each iteration of the HPE Chambolle-Pock method need a number of inner iterations which need one application of $H$ and $H^{T}$ each, we also report the distance to optimality over the number of applications of $H$ and $H^{T}$ for a fair comparison of computation effort. This also reflects the performances in terms of computational time.

We observe that the objective value of HPE CP and implicit CP when viewed over outer iterations are indistinguishable from each other, although HPE CP uses less CG iterations per outer iteration.
This can be seen in the plots on the right where we counted the number of applications with $H$ and $H^{T}$ is all methods (which are the dominant operations in all algorithms).
For the second run with the looser requirements, smaller $\lambda$ and different scaling (a larger primal stepsize $\tau$) HPE always needed just one iteration of CG to reach the necessary accuracy while the implicit Chambolle-Pock needed between $6$ and $11$ CG iterations. In general, we observed that a larger $\tau$ makes the approximate solution of the resolved more difficult and resulted in slightly more CG iterations.
We were not able to tune the explicit CP as well as the Condat-Vu method to achieve comparable performance in both runs.

\begin{figure}[htb]
  \centering
  \includegraphics[height=5cm]{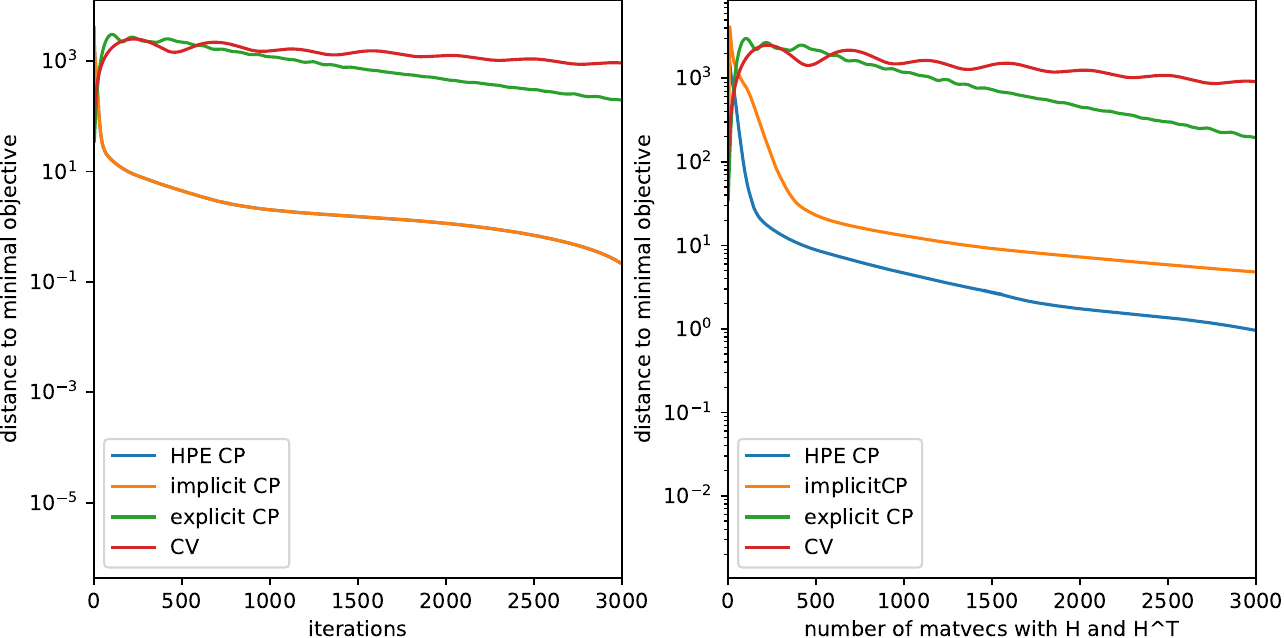}
  \caption{Objective value over iterations and objective value over time for the first run of experiment 1 described in Section~\ref{sec:cp-experiment}.}
  \label{fig:experiment1-CP-1}
\end{figure}

\begin{figure}[htb]
  \centering
  \includegraphics[height=5cm]{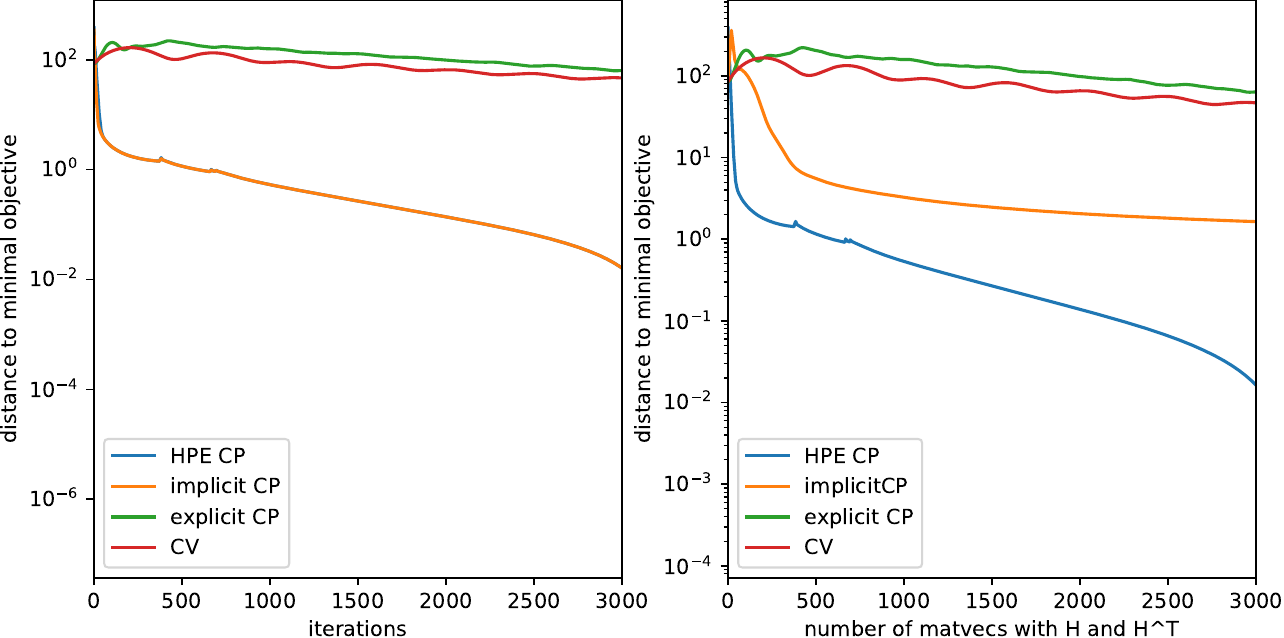}
  \caption{Objective value over iterations and objective value over time for the second run of experiment 1 described in Section~\ref{sec:cp-experiment}.}
  \label{fig:experiment1-CP-2}
\end{figure}

In a second experiment we considered $m=1000$ and $n=4000$ and a similar setup, but used a more ill-conditioned matrix where the singular values are $\sigma_{i} = (1 - \tfrac{i-1}{m-1})^{5}$ (resulting in a condition number of about $2.12\cdot10^{15}$). Algorithmic parameters have been chosen as follows: A small regularization parameter$\lambda=0.1$, and a large tolerance $\sigma=0.99$. The scaling was again set to optimize performance for all methods (resulting in $\kappa=0.5$ for all methods). The results are shown in Figure~\ref{fig:experiment2-CP}.

\begin{figure}[htb]
  \centering
  \includegraphics[height=5cm]{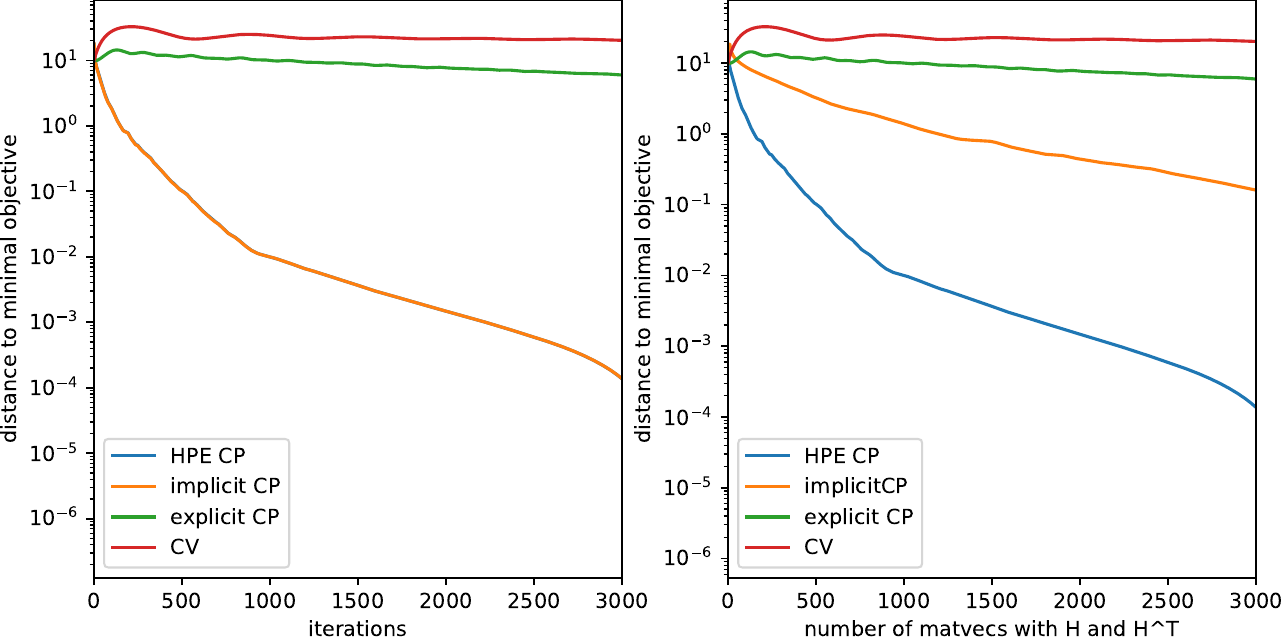}
  \caption{Objective value over iterations and objective value over time for experiment 2 described in Section~\ref{sec:cp-experiment}.}
  \label{fig:experiment2-CP}
\end{figure}

We see that the improvement of HPE over the implicit CP is even more significant. In this case HPE CP terminates the CG iteration always after the first step while the implicit CP which uses a fixed tolerance needs between about 13 iterations in the beginning and 4 iteration in later stages.

\subsection{Davis-Yin}
\label{sec:dy-experiment}

We present here a numerical example to show the applicability of Algorithm \ref{alg:HPE_DY} in practice. We consider $H\in \R^{m\times n}$, $f\in \R^m$, $\lambda_1,\lambda_2 >0$ and the problem

\begin{equation}\label{eq:HPE_DY_experiment}
 \min_{x \in \R^n} \frac{1}{2}\|H x - f\|^2 + \lambda_1\|x\|_1 + \lambda_2 L_\delta(Dx),
\end{equation}
where $D\in\R^{k\times n}$ is some dictionary matrix (we use again a finite difference matrix just as an example) and $L_\delta$ is the Huber loss defined by $L_\delta(y) = \sum_{i=1}^k h_\delta(y_i)$ with
\[h_\delta(y_i)=\begin{cases}\frac{1}{2}|y_i|^2 & |y_i| \leq \delta\\
\delta (|y_i|-\frac{1}{2}\delta) & |y_i|>\delta\end{cases}\]
for some $\delta >0$. The Huber loss is usually adopted as a smooth approximation of the $1$-norm and in our case produce the term $L_\delta(Dx)$ which is know as Huber-TV regularization~\cite{werlberger2009anisotropic,hintermuller2013nonconvex}. Notice that the gradient of $L_\delta$ is $\nabla L_\delta (y) = (\partial_{y_1} h_\delta (y_1),\dots, \partial_{y_k} h_\delta (y_k))$, where \begin{equation*}\label{eq:gradient_Huber}
\partial_{y_i} h_\delta(y_i)=\begin{cases} y_i & |y_i| \leq \delta\\
\delta \, \text{sign}(y_i) & |y_i|>\delta,\end{cases}\end{equation*}
and thus $\nabla L_\delta$ is $1$-Lipschitz and $L_\delta \circ D$ is $\|D\|^2$-smooth. The parameters $\lambda_1$ and $\lambda_2$ serve as regularization parameters.

There are numerous approaches to solve this problem with splitting methods, and again, the term $\tfrac12\norm{Hx-f}_2^2$ can be treated in different ways.
\begin{itemize}
\item Doing implicit steps, i.e. evaluating the respective resolvent of $\tfrac12\norm{Hf-f}_{2}^{2}$ exactly. This amount to solving a linear system with the operator $I + \tau H^{T}H$. As in the previous section, we will do this applying the method of conjugate gradients (CG) for the solution of the linear system up to a certain fixed tolerance. This is done in what we call \emph{implicit Davis-Yin} and explain the details below.
\item Doing approximate implicit steps by using the method of conjugate gradients (CG) for the solution of the linear systems within the HPE framework as described in Algorithm~\ref{alg:HPE_DY}.
\item Doing forward steps: Since the term is $L$-smooth, one can do forward steps of the form $H^{T}(Hx-f)$. This is possible making use for example of the forward-backward (FB) algorithm (or proximal-gradient) \cite{Tseng_FB,Rockafellar_FB,Combettes_FB,Attouch_FB}.
\end{itemize}

\subsubsection{Davis-Yin with implicit steps}\label{sec:implicit-DY}
 We put the optimization problem \eqref{eq:HPE_DY_experiment} in correspondence with \eqref{eq:3op_problem} by using Remark \ref{rem:DY-functions} and setting $A_1 x=H^T(H x-b)$, $A_2 x=\lambda_1 \partial \|x\|_1$ and $B x= \lambda_2 \nabla (L_\delta \circ D ) (x) = \lambda_2 D^T \nabla L_\delta ( D x)$. We apply the standard Davis-Yin method \cite{davis2017} with this choice of operators, resulting in the scheme
 \begin{equation}\label{eq:implicitDY-iteration}
  \begin{aligned}
    x_1^{k+1} & = (I+\gamma H^{T}H)^{-1}( w^k- \gamma H^{T}f))\\
    x_2^{k+1} & = \operatorname{soft}_{[-\gamma \lambda_1,\gamma \lambda_1]} (2x_1^{k+1}-w^k-\gamma \lambda_2 D^T \nabla L_\delta ( D x_1^{k+1}))\\
    w^{k+1} & = w^k + \frac{x_2^{k+1}-x_1^{k+1}}{1+\alpha}
  \end{aligned}
\end{equation}
where $\gamma \in (0,2/\beta)$ is the stepsize (for some constant $\beta$), $\alpha = \frac{\gamma \beta}{4-\gamma \beta} $ and the soft tresholding operator \[\operatorname{soft}_{[-\eta,\eta]}(x) = \text{sign}(x)\max(|x|-\eta,0)\] is computed componetwise. In order to guarantee convergence of the Davis-Yin algorithm, the constant $\beta$ has to be greater than the cocoercivity constant of $B$, which coincides with its Lipschitz constant, since $B$ is a gradient of a convex function (see~\cite[Theorem 18.15]{BCombettes}). In the experiments we estimated from above this constant, using the fact that $\|D\|\leq 2$ and setting $\beta = 4 \lambda_2 \geq \lambda_2 \|D\|^2$. Then, we selected the stepsize $\gamma = \frac{1}{\beta}$.
 
\subsubsection{HPE Davis-Yin}\label{sec:HPE-DY}
To avoid the exact solution of the linear system in the Davis-Yin method of the previous section, we can apply the HPE framework from Section~\ref{sec:DY}, i.e. Algorithm~\ref{alg:HPE_DY}. Instead of evaluating the resolvent $J_{\tau A_{1}}$ exactly
we apply the method of conjugate gradients to improve the pair $(\tilde{x_1}^{k+1},a_1^{k+1})$. Note that for a given candidate $\tilde{x}_1^{k+1}$ there is a unique $a_1^{k+1} = H^{T}(H\tilde{x}_1^{k+1}-f)\in A_{1}\tilde{x}_1^{k+1}$.

For the HPE DY method we need $\gamma \in \left(0,\frac{2}{\beta}\right)$ and we take again $\gamma = \frac{1}{\beta}$, with $\beta$ estimating the value of the cocoercivity constant of $B$ from above as seen in Section~\ref{sec:implicit-DY}.

\subsubsection{Forward-Backward}\label{sec:FB}
Introducing the operators $A,\tilde B$ defined by $A x=\lambda_1 \partial \|x\|_1$ and $\widetilde Bx = H^T(H x-f) + \lambda_2 D^T \nabla L_\delta ( D x)$, respectively, the problem \eqref{eq:HPE_DY_experiment} can be written as
\[\text{find } x \in \R^n \quad \text{such that } \quad  0 \in Ax+ \tilde Bx.\]
To this problem, one can apply the classical forward-backward algorithm (or proximal-gradient) \cite{Tseng_FB,Rockafellar_FB,Combettes_FB,Attouch_FB}. The resulting scheme reads as
 \begin{align}\label{eq:implicitFB-iterations}
  \begin{split}
    x^{k+1} & =  \operatorname{soft}_{[-\gamma \lambda_1,\gamma \lambda_1]}(x^k- \gamma H^T(H x^k-f - \gamma \lambda_2 D^T \nabla L_\delta ( D x^k))
  \end{split}
\end{align}

Also here, the stepsize $\gamma$ has to be lower than a certain quantity. In particular, to guarantee convergence, we estimated from above the cocoercivity constant of the operator $\tilde B$, by setting $\tilde \beta = \|H\|^2+4 \lambda_2 \geq \|H\|^2+\lambda_2\|D\|^2 $ (using again the fact that $\|D\|\leq 2$), and then selected $\gamma = 1/\tilde \beta$.

\subsubsection{Experimental results}
\label{sec:DY-results}
For the experiments, we used $m=n=2000$ and generated an ill-conditioned matrix $H = U\Sigma V^{T}$ with two random orthonormal matrices $U,V$ and a diagonal matrix $\Sigma$ with diagonal entries $\sigma_{i} = \tfrac12 + \tfrac12\cos\left(\pi\tfrac{i-1}{m-1}\right)$, $i=1\dots,m$. In this way the singular values of $H$ decay from $1$ to $0$ along a cosine curve, so there is no big cluster of singular values. The resulting condition number was about $4.7\cdot10^{8}$. As matrix $D$ we take the $(n-1)\times n$ matrix that contains the first finite differences. As $f$ we take $f= Hx^{\dag} + \eta$ where $x^{\dag}$ is a piecewise constant signal and $\eta$ is Gaussian noise.
We performed three runs of an experiment comparing the performance of the following algorithms:
\begin{itemize}
\item implicit Davis-Yin (implicit DY) from Section \ref{sec:implicit-DY},
\item inexact Davis-Yin with the HPE framework (HPE DY) from Algorithm~\ref{alg:HPE_CP} as described in Section \ref{sec:HPE-DY},
\item the forward-backward method (FB) from Section \ref{sec:FB}.
\end{itemize}
 
 For the implicit Davis-Yin we also used the CG method to solve the linear system. However, we used a fixed relative tolerance of $10^{-8}$ for the residual to terminate the method (as it is usual for CG). In all cases where the CG method has been used it has been initialized with the current iterate $x$ as warmstart.\\

We made three runs with the following setups:
\begin{itemize}
\item First run: Regularization parameters $\lambda_1=0.001$, $\lambda_2=0.1$ and for HPE DY we demanded low accuracy by setting $\sigma=0.99$.
\item Second run: Smaller regularization parameters $\lambda_1=0.0001$, $\lambda_2=0.1$ and again accuracy $\sigma=0.99$.
\item Third run: Smaller regularization parameters $\lambda_1=0.0001$, $\lambda_2=0.01$ and again accuracy $\sigma=0.99$.\\
\end{itemize}

We report the distance to the minimal objective value (computed with a larger number of iterations) in log-scale over iterations in Figures~\ref{fig:experiment11-DY},~\ref{fig:experiment12-DY} and~\ref{fig:experiment13-DY} and since each iteration of the HPE Davis-Yin method need a number of inner iterations which need one application of $H$ and $H^{T}$ each, we also report the distance to optimality over the number of applications of $H$ and $H^{T}$ for a fair comparison of computation effort. This also reflects the performances in terms of computational times.

We observe that the objective value of HPE DY and implicit DY when viewed over outer iterations are similar to each other, although HPE DY uses less CG iterations per outer iteration. This can be seen in the plots on the right where we counted the number of applications with $H$ and $H^{T}$ in all methods (which are the dominant operations in all algorithms).
In all the runs, HPE always needed just one or two iteration of CG to reach the necessary accuracy while the implicit Davis-Yin needed between $6$ and $11$ CG iterations. We were not able to tune the FB method to achieve comparable performance in the runs.

\begin{figure}[htb]
  \centering
  \includegraphics[height=5cm]{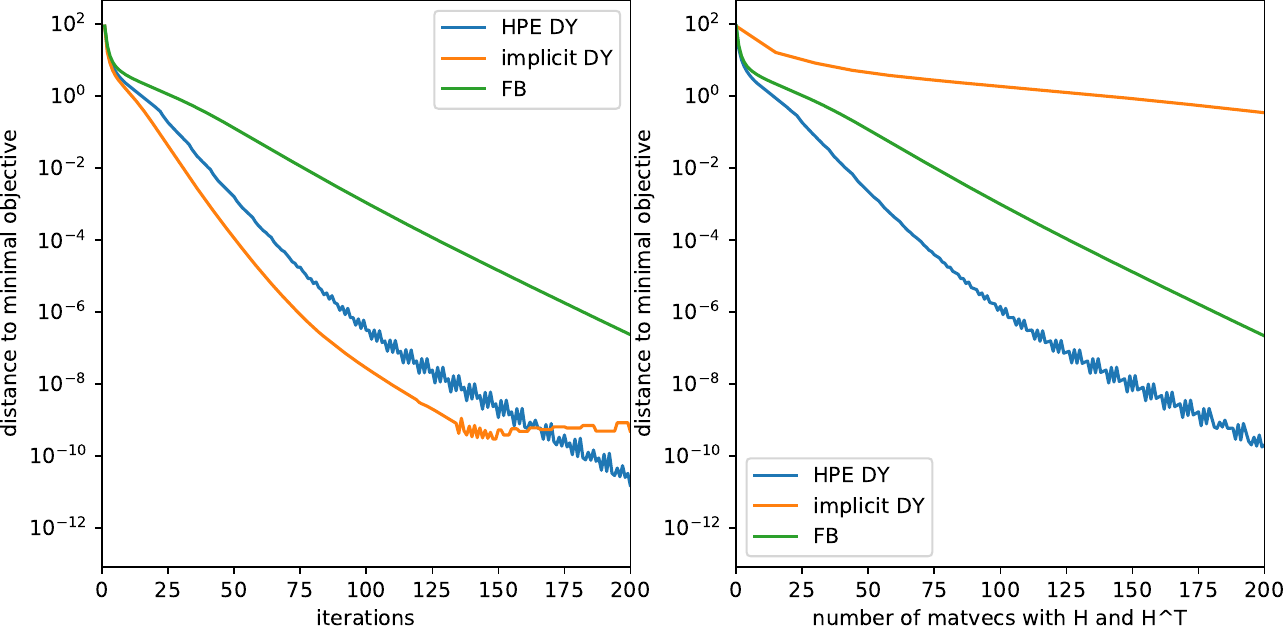}
  \caption{Objective value over iterations and objective value over time for the first run described in Section~\ref{sec:dy-experiment}.}
  \label{fig:experiment11-DY}
\end{figure}

\begin{figure}[htb]
  \centering
  \includegraphics[height=5cm]{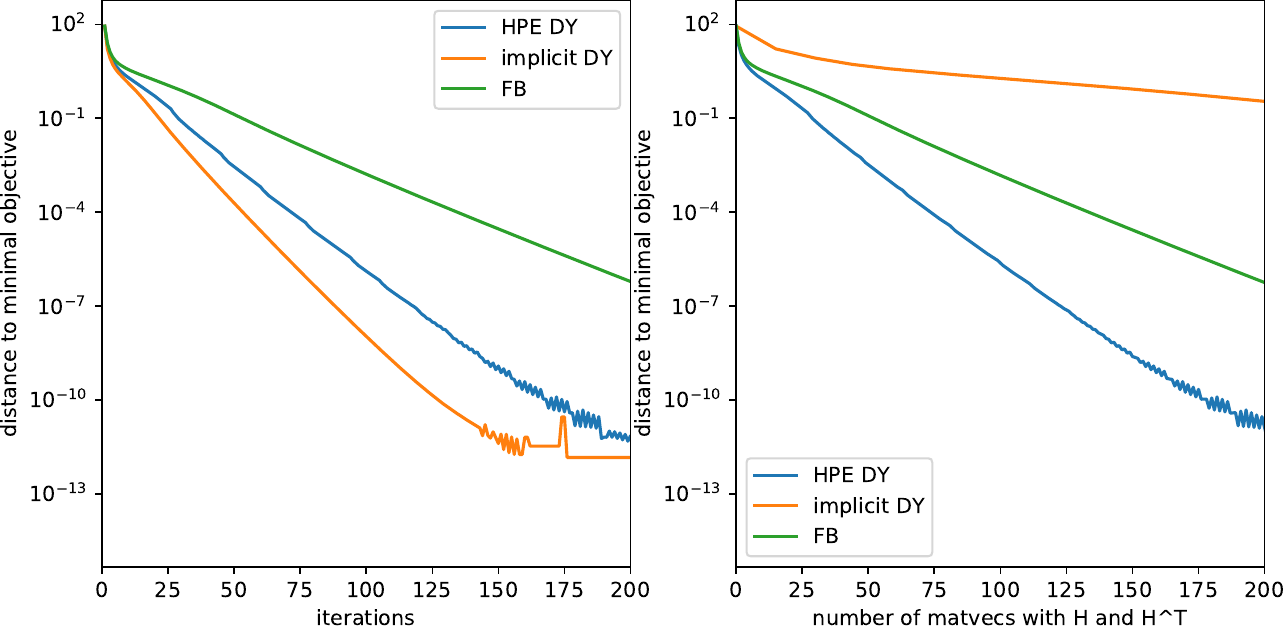}
  \caption{Objective value over iterations and objective value over time for the second run described in Section~\ref{sec:dy-experiment}.}
  \label{fig:experiment12-DY}
\end{figure}

\begin{figure}[htb]
  \centering
  \includegraphics[height=5cm]{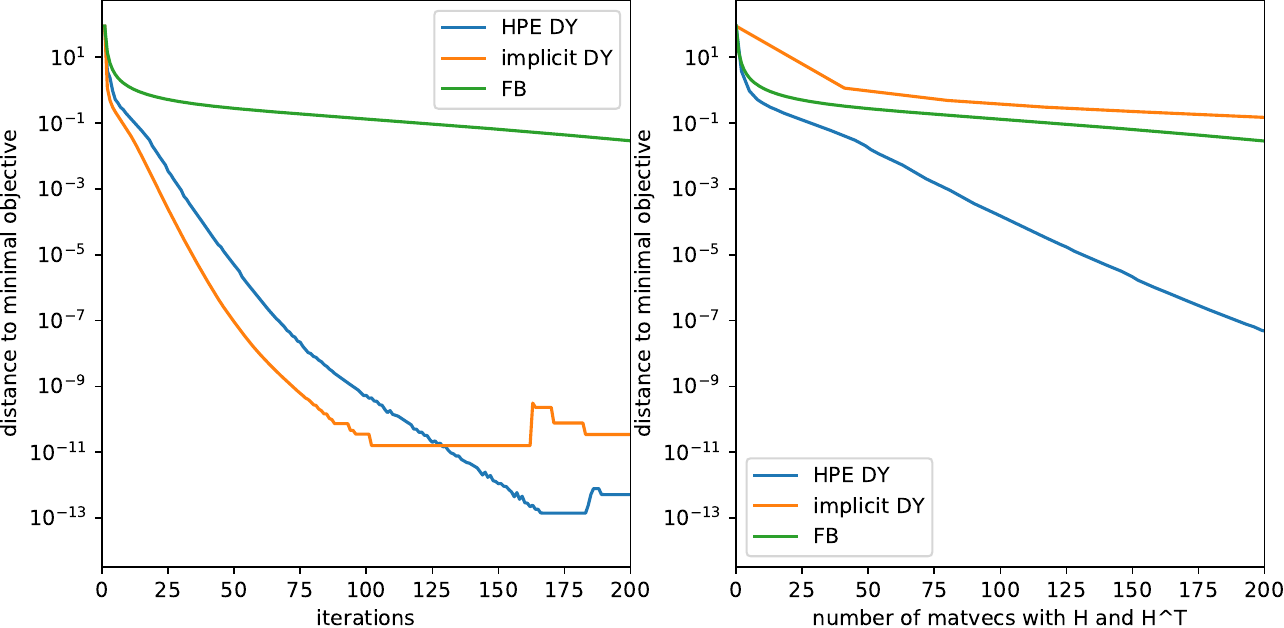}
  \caption{Objective value over iterations and objective value over time for the third run described in Section~\ref{sec:dy-experiment}.}
  \label{fig:experiment13-DY}
\end{figure}

\section*{Declarations}

\noindent \textbf{Data Availability Statement.} All data used in this paper is generated synthetically. The code is available at \url{https://github.com/EmanueleNaldi/Degenerate-HPE}.

 \bigskip

\noindent \textbf{Conflict of interest.} The authors declare that there are no conflicts of interest regarding the publication of this paper.

\bigskip

\noindent \textbf{Acknowledgments.} Part of this work was realized while M. M. A. was visiting the Institute of Analysis and Algebra, TU Braunschweig, Germany, under the financial support of CAPES-PRINT/UFSC Institutional Internationalization Program, which is gratefully acknowledged. The work of M. M. A. is partially supported by CNPq grant 308036/2021-2. D. A. L. acknowledges support by the Alexander von Humboldt Foundation for a visit at the Federal University of Santa Catarina (UFSC). This work has received funding from the European Union’s Framework Programme for Research and Innovation Horizon 2020 (2014--2020) under the Marie Skłodowska--Curie Grant Agreement No. 861137. E. N. acknowledges the support of the US Air Force Office of Scientific Research (FA8655-22-1-7034). The research by E. N. has been supported by the MUR Excellence Department Project awarded to Dipartimento di Matematica, Università degli Studi di Genova, CUP D33C23001110001. E. N. is a member of the ``Gruppo Nazionale per l’Analisi Matematica, la Probabilità e le loro Applicazioni" (GNAMPA) of the Istituto Nazionale di Alta Matematica (INdAM).


\bibliography{references}

\end{document}